\documentclass[12pt]{amsart}
\usepackage[utf8]{inputenc}
\usepackage[english]{babel}
\usepackage{amsmath, amssymb, amsfonts, enumerate, amsthm, mathtools, leftindex, xcolor}
\usepackage{mathrsfs}
\usepackage{hyperref}
\usepackage{newtxtext}
\usepackage[margin=2cm]{geometry}
\usepackage{orcidlink}

\newcommand{\assign}{:=}
\newcommand{\backassign}{=:}
\newcommand{\lto}{\longrightarrow}

\newcommand{\tmop}[1]{\ensuremath{\operatorname{#1}}}

\newenvironment{enumeratealpha}{\begin{enumerate}[a{\textup{)}}] }{\end{enumerate}}

\newtheorem{theorem}{Theorem}[section]
\newtheorem{lemma}[theorem]{Lemma}
\newtheorem{proposition}[theorem]{Proposition}
\newtheorem{corollary}[theorem]{Corollary}
\theoremstyle{definition}
\newtheorem{remark}[theorem]{Remark}
\newtheorem{example}[theorem]{Example}
\theoremstyle{definition}
\newtheorem{definition}[theorem]{Definition}

\title{Multipoint Schwarz--Pick Lemma for the quaternionic case}
\author [Cinzia Bisi]{Cinzia Bisi \orcidlink{0000-0002-4973-1053}}
\address{C. Bisi: Dipartimento di Matematica e Informatica, Universit\`a degli Studi di Ferrara, Via N. Machiavelli 30, 44121, Ferrara, Italia.
\newline \emph{ORCID link}: \orcidlinkf{0000-0002-4973-1053}} 
\email{\href{mailto:bsicnz@unife.it}{bsicnz@unife.it}}

\author[Davide Cordella]{Davide Cordella \orcidlink{0009-0001-1163-0046}}
\address{D. Cordella: Dipartimento di Ingegneria Industriale e Scienze Matematiche, Universit\`a Politecnica delle Marche, Via Brecce Bianche 12, 60131, Ancona, Italia. \newline \emph{ORCID link}: \orcidlinkf{0009-0001-1163-0046}} 
\email{\href{mailto:d.cordella@staff.univpm.it}{d.cordella@staff.univpm.it}}
\thanks{The two authors were partially supported by GNSAGA of INdAM and by PRIN \emph{Varietà complesse: geometria, topologia e analisi armonica}}

\date{}

\subjclass[2020]{30G35}
\keywords{Schwarz--Pick Lemma, slice regularity, M\"obius transformations, hyperbolic difference quotient, Nevanlinna--Pick interpolation}

\begin{document}

\setlength{\parindent}{0pt}

\begin{abstract}
Following ideas by Beardon, Minda and Baribeau, Rivard, Wegert in the context of the complex Schwarz--Pick Lemma, we use iterated hyperbolic difference quotients to prove a quaternionic multipoint Schwarz--Pick Lemma, in the context of the theory of slice regular functions. As applications, we obtain quaternionic Dieudonn\'e and Goluzin estimates. Finally, an algorithm for the construction of (Nevanlinna--Pick) interpolating slice regular functions with real nodes is provided as a byproduct of the quaternionic multipoint Schwarz--Pick Lemma.
\end{abstract}
\maketitle
\section*{Introduction}

In complex analysis, holomorphicity plays an important role in the study of the intrinsic geometry of the unit disk $\mathbb D=\{z\in\mathbb C: |z|<1\}$, thanks to the Schwarz--Pick Lemma,  which asserts the following:

\begin{lemma}[Schwarz--Pick Lemma]
Let $f:\mathbb D\lto\mathbb D$ be a holomorphic function. Let $z_0\in\mathbb D$. Then for any $z\in\mathbb D$ it is
\begin{equation}\label{eq:spl_complex}
	\left|\frac{f(z)-f(z_0)}{1-\overline {f(z_0)}f(z)}\right|\leq \left|\frac{z-z_0}{1-\overline{z_0}z}\right|;
\end{equation}
moreover
\begin{equation*}
\frac{|f'(z_0)|}{1-|f(z_0)|^2}\leq\frac{1}{1-|z_0|^2}.
\end{equation*}
Inequalities are strict for $z\neq z_0$, unless $f$ is an automorphism of the disk.
\end{lemma}

This lemma was formulated at the beginning of the 20\textsuperscript{th} century by Pick \cite{pick1} and it has also a more geometric interpretation: a holomorphic self-map of $\mathbb D$, if it is not an automorphism, is a contraction with respect to the Poincaré metric on $\mathbb D$, $ds^2=(1-|z^2|)^{-2} dz d\overline z$ and in the automorphism case it is an isometry. This result turns out to be a special case of Nevanlinna--Pick interpolation Theorem, which gives a necessary and sufficient condition for the existence of a holomorphic function $f:\mathbb D \lto \mathbb D$ such that $f(z_m)=w_m$ for all $m=1,\ldots,n$, given $n$ distinct points $z_1,\ldots,z_n\in\mathbb D$  and $w_1,\ldots w_n\in\mathbb D$. The condition is that the {Pick matrix}
\begin{equation*}
	P=\left( \frac{1-w_a\overline{w_b}}{1-z_a\overline{z_b}} \right)_{a,b=1,\ldots,n}
\end{equation*}
is positive semidefinite. It is straightforward to see that when $n=2$ this condition reduces to inequality \eqref{eq:spl_complex}.

The Schwarz--Pick Lemma has been extended in various ways, in many years and by many authors. Among the others, Beardon and Minda \cite{BeardonMinda04} 
have introduced the following hyperbolic difference quotient

\[ f^\star_{z_0}(z)\assign \begin{cases} \frac{[f(z),f(z_0)]}{[z,z_0]} &z\neq z_0\\ f'(z_0)\frac{1-|z_0|^2}{1-|f(z_0)|^2} & z=z_0, \end{cases}\qquad \text{where }[z,w]=\frac{z-w}{1-\overline{w}z} \]

and with this quantity they have got a `three points' version of the Schwarz--Pick Lemma

\begin{theorem}[{\cite[Thm. 3.1]{BeardonMinda04}}]
Let $f:\mathbb D\lto\mathbb D$ be holomorphic but not an automorphism. Let $z_0,z_1,z_2\in\mathbb D$. Then
\begin{equation*}
	\left|\frac{f^\star_{z_0}(z_1)-f^\star_{z_0}(z_2)}{1-\overline{f^\star_{z_0}(z_2)}f^\star_{z_0}(z_1)}\right|\leq \left|\frac{z_1-z_2}{1-\overline{z_2}z_1}\right|.
\end{equation*}
The inequality is strict for $z_1\neq z_2$ unless $f$ is a complex Blaschke product of degree two.
\end{theorem}

From this result they have also obtained easy proofs of classical estimates for derivatives of holomorphic self-maps of the unit disk, assuming that the origin is a fixed point.

Baribeau, Rivard and Wegert \cite{BaribeauRivardWegert} have considered iterated hyperbolic difference quotients $f^{\{n\}}_{z_1,\ldots,z_n}=(f^{\{n-1\}}_{z_1,\ldots,z_{n-1}})^\star_{z_n}=(((f^\star_{z_1})^\star_{z_2})^\star_{z_3}\ldots)^\ast_{z_n}$, getting an analogous theorem with more than three points. Moreover, they used these iterated differences to give simpler conditions for the $n$ points Nevanlinna--Pick interpolation problem, and also a constructive algorithm for the solutions which is based on the works of Schur \cite{Schur1917} and Nevanlinna \cite{Nevanlinna1919}. Iterated hyperbolic difference quotients have been also considered in \cite{ChoKimSugawa12} in order to obtain estimates for higher order derivatives and recently in \cite{abate_julia} to generalize Julia's Lemma.

Over the skew-field of quaternions $\mathbb H$, which contains $\mathbb C$, the notion of slice regularity takes the place of holomorphicity. Slice regular functions have been introduced by Gentili and Struppa \cite{advances}, developing a regularity concept that was originally formulated by Cullen \cite{cullen}. It is based on the fact that it can be given a `book structure' to the skew-field of quaternions $\mathbb H=\bigcup_{I\in\mathbb S} \mathbb C_I$,
where $\mathbb S$ is the sphere of \emph{imaginary units} $\{q\in\mathbb H:q^2=-1\}$ and $\mathbb C_I$ are the complex subspaces $\{x+y I\in\mathbb H:x,y\in\mathbb R\}$, which are known as \emph{complex slices} and which can be identified with complex planes having the real line $\mathbb R$ in common. A function $f:\Omega\lto\mathbb H$ defined on a domain in $\mathbb H$ is said to be slice regular if and only if its restriction to any slice $\mathbb C_I$ is holomorphic, i.e., the differential operator $\overline\partial_I=1/2(\partial_x+I\partial_y)$ vanishes on $f\mid_{\Omega\cap\mathbb C_I}$. If so, one can define a derivative $\partial_C f$ (the \emph{regular}, or \emph{Cullen}, \emph{derivative}) in this way: $\partial_C f(q)=\partial_I (f\mid_{\Omega\cap\mathbb C_I})(q)$ for any $q\in\Omega\cap\mathbb C_I$, where $\partial_I$ is the operator $1/2(\partial_x-I\partial_y)$. It is important to observe that the common pointwise product of functions does not preserve slice regularity; however, one could define a different product on spaces of slice regular functions, the $\ast$-product $f\ast g$, which allows to define on them the structure of a (noncommutative) algebra. Using this operation together with the $\ast$-inversion $f\mapsto f^{-\ast}$, one can define slice regular M\"obius transformations on the quaternionic unit ball $\mathbb B=\{q:\mathbb H: |q|<1\}$, that is to say the maps of the form
\[
(1-q \overline{p})^{-\ast}\ast (q-p)u, \quad p\in\mathbb B, u\in\partial \mathbb B.
\]
These above are the unique slice regular bijective functions mapping $\mathbb B$ onto itself. When $u=1$, we denote the map above by $\mathcal M_p$. Analogously, one can give a definition of slice regular finite Blaschke products on the unit ball $\mathbb B$, which can be characterized in a similar way as in the complex framework.

All the facts we need from the theory of slice regularity are reviewed in Section \ref{sec:prel}, with a focus on slice regular M\"obius transformations (following \cite{volumeindam,moebius}) and regular Blaschke products on $\mathbb B$. 
The main difference between the
classical complex setting and the quaternionic frame is essentially due to the contrasts between the theory of slice regularity with respect to the classical holomorphicity, which are stressed in Section \ref{sec:prel}. We recall, for instance, the fact that we cannot use compositions of
functions as in the complex case, but we need to consider an action of a matrix group on slice
regular functions. These discrepancies, due essentially to the non-commutativity of the quaternions, have some consequences in trying to apply arguments from complex analysis to this novel framework: we will need to choose some points along the real line (which is incidentally the center of the skew-field of quaternions). For further information {about the theory of quaternionic slice regularity} we refer to the  recent monograph by Gentili, Stoppato and Struppa \cite{librospringer2}. 

Bisi and Stoppato \cite{BSIndiana} gave a quaternionic version of the Schwarz--Pick Lemma in the framework of this theory. Analogously to what happens in the complex case, in Section \ref{sec:spl} we extend the notion of hyperbolic difference quotient $f^\star_p$ for slice regular functions from $\mathbb B$ to itself, thus the quaternionic Schwarz--Pick Lemma leads to a three-points Schwarz--Pick Lemma:
\begin{theorem} Let $\mathbb B=\{q\in\mathbb H:|q|<1\}$ be the open quaternionic unit ball. Suppose $f : \mathbb{B} \lto \mathbb{B}$ is a slice regular function which is not bijective. Let $p,s\in\mathbb B$. Then for all $q\in\mathbb B$
	\begin{equation*}
		\left| (\mathcal M_{f^\star_p(s)}\bullet f^\star_p)(q)\right|\leq \left|\mathcal M_s(q)\right|,
	\end{equation*} where $\mathcal M_{f^\star_p(s)}\bullet f^\star_p$ stands for $(1-f^\star_p\ast \overline{f^\star_p(s)})^{-\ast}\ast(f^\star_p-f^\star_p(s))$.
Equality holds for some $q\in\mathbb B\setminus\{s\}$ (and hence for any $q\in\mathbb B$) if and only if $f$ is a slice regular Blaschke product of degree two.
	
\end{theorem}

In Section  \ref{sec:conseq} we show some consequences of these results. In particular, following \cite{BeardonMinda04} and restricting to the case of a fixed point at the origin, by the connection between the {slice regular derivative} $\partial_C f(p)$ and the {hyperbolic derivative} $f^h(p)\assign f^\star_p(p)$ for slice regular self-maps on $\mathbb  B$, we get the quaternionic version of some inequalities well established in classical complex analysis, such as Dieudonn\'e's or Goluzin's estimates (assuming $f(0)=0$).

The definition of iterated hyperbolic difference quotients in the quaternionic setting and the consequent multipoint Schwarz--Pick Lemma are presented in Section \ref{sec:iterated}.

In Section \ref{sec:np} we consider the Nevanlinna--Pick interpolation problem with $n$ points in the quaternionic unit ball $\mathbb B$, looking for a slice regular interpolating function $f:\mathbb B\lto\mathbb B$ for the prescribed nodes and values. 
Here we give a generalization to the quaternionic setting of the conditions and the algorithm presented in \cite{BaribeauRivardWegert} for the complex case, in terms of hyperbolic difference quotients. As recalled before, since the quaternionic setting is  wilder than the complex one and not commutative, we have to restrict the problem to the case of \emph{real} nodes. Given $n$ distinct real points $r_1,\ldots, r_n\in(-1,1)$ and $n$ values $s_1,\ldots,s_n$ in the quaternionic unit ball $\mathbb B$, one defines by iteration quantities $Q_{\kappa}^\ell$ for $\kappa=1,\ldots,n-1$ and $\ell=\kappa+1,\ldots,n$ which are generically expressed as quotients of M\"obius maps, starting from the given nodes and values. Then we have a criterion in terms of the last term $Q_{n-1}^n$:
\begin{theorem}
There exists a slice regular function $f:\mathbb B\lto \mathbb B$ such that $f(r_m)=s_m$ for all $m=1,\ldots,n$ if and only $Q_{n-1}^n$ is a quaternion of modulus less or equal to one. If $|Q_{n-1}^n|<1$, then there is an infinite family of solutions, while if $|Q_{n-1}^n|=1$ there exists a unique solution given by a regular Blaschke product.
\end{theorem}
The terms $Q_\kappa^\ell$ correspond to values of the iterated difference quotients of any solution of the problem: $Q_\kappa^\ell=f^{\{\kappa\}}_{r_1,\ldots,r_\kappa}(r_\ell)$. The algorithm used for proving this result also provides explicit formulas for the solutions.

When the nodes are not all real, the algorithm does not work anymore: we give some evidence in the last part of Section \ref{sec:np}.

A quaternionic general Nevanlinna--Pick Theorem in terms of the positive semi-definition of the associated Pick matrix has already been proved by Alpay, Bolotnikov, Colombo and Sabadini in \cite{alpaybolotnikovcolombosabadini} with a different approach based on linear algebra and functional analysis techniques. Although this result is more general, it gives less explicit formulas for the solutions.

\section{Preliminaries}\label{sec:prel}

\subsection{Slice regular functions}\mbox{}

Let $\mathbb H=\mathbb R+\mathbb Ri+\mathbb Rj+\mathbb Rk$ denote the skew-field of quaternions. Let $\mathbb S$ be the sphere of imaginary units: $\mathbb S=\{q\in\mathbb H:q^2=-1\}=\{q\in\mathbb H:\tmop{Re} (q)=0,|q|=1\}$. For any $I\in\mathbb S$, let $\mathbb C_I$ denote the complex subspace $\{x+y I\in\mathbb H:x,y\in\mathbb R\}$, which we will call a \emph{complex slice}. We observe that $\mathbb C_{-I}=\mathbb C_I$ and that $\mathbb H$ is the union of such slices:
\begin{equation*}
	\mathbb H=\bigcup_{I\in\mathbb S} \mathbb C_I.
\end{equation*}
Moreover, all the slices have the real line in common: $\bigcap_{I\in\mathbb S} \mathbb C_I=\mathbb R$. This decomposition of the quaternionic space is sometimes known as \emph{book structure}.
\begin{definition}
Let $\Omega\subseteq\mathbb H$ be a domain. For any $I\in\mathbb S$, let $\Omega_I\assign \Omega\cap\mathbb C_I=\{x+yI\in\Omega:x,y\in\mathbb R\}$. A function $f:\Omega\lto\mathbb H$ is said to be \emph{(left) slice regular} on $\Omega$ ($f\in\mathcal{SR}(\Omega)$) if it has continuous partial derivatives and any restriction $f_I=f\mid_{\Omega_I}$ for $I\in\mathbb S$ is holomorphic, in the sense that $\overline\partial_I f_I(x+y I)\equiv 0$ on $\Omega_I$, where $\overline{\partial}_I$ is the differential operator
\[
 \overline\partial_I\assign \phi_I\mapsto \frac{1}{2}\left(\partial_x\phi_I+I\partial_y\phi_I\right),\qquad\phi_I:\Omega_I\lto\mathbb H.
\]
\end{definition}
A trivial consequence of this definition, observing that $\mathbb H=\mathbb C_I + \mathbb C_I J$ for any $I\in\mathbb S$ and $J\in\mathbb S$ orthogonal to $I$ with respect to the standard scalar product of $\mathbb R^4$, is the following:
\begin{lemma}[Splitting Lemma, {\cite[Lemma 1.3]{librospringer2}}]
	Let $f:\Omega\lto \mathbb H$ be slice regular on a domain $\Omega\subseteq \mathbb H$. For any $I\in\mathbb S$ and $J\in\mathbb S$ such that $J\perp I$, there exist two holomorphic functions $F,G: \Omega_I\lto\mathbb H$, depending on $I$ and $J$, such that the restriction $f_I=f\mid_{\Omega_I}$ writes down as
	\[
	f_I(z)=F(z)+G(z)J\qquad\text{for all }z=x+yI\in\Omega_I.
	\]
\end{lemma}
\begin{definition}
Let $\Omega\subset\mathbb H$ be a domain and let $f:\Omega\lto\mathbb H$ be a slice regular function. For any $I\in\mathbb S$ we set $\partial_I f_I(x+yI)\assign \frac{1}{2}\left[\partial_x f_I (x+yI)-I\partial_yf_I(x+yI)\right]$. The \emph{regular derivative} (or \emph{Cullen derivative}) of $f$ is defined as
\begin{equation*}
\partial_C f(q)\assign \partial_I f_I (q)\quad \text{ if }q\in\Omega_I.
\end{equation*}

The definition is well-posed since if $x_0\in\bigcap_{I\in\mathbb S} \Omega_I=\Omega\cap\mathbb R$, then $\partial_I f_I (x_0)=\partial_x f(x_0)$ for any imaginary unit $I$.

It is clear by the definition that the Cullen derivative of a slice regular function $f$  is again slice regular on the same domain, hence one can define higher (Cullen) derivatives $f^{(n)}(q)$ as well:
\begin{equation*}
	f^{(n)}(q)\assign (\partial_I)^n f_I (q)\quad \text{ if }q\in\Omega_I.
\end{equation*}
\end{definition}

\begin{example}
	It is easy to check that all polynomial functions \emph{with right quaternionic coefficient}
	\[ f(q)= q^N a_N + q^{N-1} a_{N-1}+\ldots+q a_1 + a_0 \]
	are slice regular functions over $\mathbb H$; moreover the power series of the form 
	\[ g(q)=\sum_{m=0}^{+\infty} q^m a_m \] are slice regular function over the ball of convergence $B(0,\varrho)\assign \{q\in\mathbb H: |q|<\varrho\}$, where $\varrho^{-1}=\limsup\limits_{m\to+\infty} |a_m|^{1/m}$.
	
	On the other side, the function $q\mapsto i\cdot q$ is not slice regular, since if we consider the slice $\mathbb C_j$, then the relative restriction is $x+yj\mapsto xi+yk$ and $\overline{\partial}_j(xi+yk)=\frac{1}{2}(i+jk)=\frac{2i}{2}=i$. 
\end{example}

\begin{remark}
	The sum $f+g$ of two slice regular functions $f,g\in\mathcal{SR}(\Omega)$ is still slice regular. 
	The last example shows that the standard pointwise product $fg$ of two slice regular functions $f,g$ is not always slice regular. Moreover, the composition of two slice regular functions $g\circ f$, if defined, is not in general slice regular: for instance $q\mapsto (q\cdot i)^2=q\cdot i\cdot q\cdot i$ is such that its restriction to $\mathbb C_j$ is $(x+yj)\mapsto (xi-yk)^2=-x^2-y^2$ and $\overline\partial_j(-x^2-y^2)=-(x+yj)\neq 0$ for all $x+yj\neq 0$.
\end{remark}
\begin{definition}
A domain $\Omega\subset \mathbb H$ is said a \emph{slice domain} if $\Omega \cap\mathbb R\neq\varnothing$ and any intersection $\Omega_I$ with a slice $\mathbb C_I$ is connected.

A domain $\Omega$ as above is \emph{axially symmetric} if for any of its points $x+yI$ (where $x,y\in\mathbb R$, $I\in\mathbb S$) the whole sphere $x+y\mathbb S\assign \{x+yJ:J\in\mathbb S\}$ is contained in $\Omega$.
\end{definition}
Making both these assumptions on the domain, we have the following remarkable representation formula for slice regular function.
\begin{proposition}[Representation Formula, {\cite[Thm. 1.16]{librospringer2}}]
	Let $f:\Omega\lto\mathbb{H}$ be a slice regular function defined on an axially symmetric slice domain
	Let $I,J,K\in\mathbb S$ such that $J\neq K$. Then
	\begin{equation*}
		f(x+yI)=(J-K)^{-1}[J f(x+yJ)-K f(x+yK)]+I(J-K)^{-1}[f(x+yJ)-f(x+yK)].
	\end{equation*}
	In particular, if we choose $K=-J$, we get
	\begin{equation}\label{eq:represent}
	f(x+yI)=\frac{1}{2}[f(x+yJ)+ f(x-yJ)]+\frac{IJ}{2}[f(x-yJ)-f(x+yJ)].
	\end{equation}
\end{proposition}
It is important to observe that formula \eqref{eq:represent} says that the values of $f$ on a slice determine the whole function.

If we consider balls centered in the origin as domains, we get that all functions which are slice regular can be written as power series:
\begin{theorem}[{\cite[Thm. 1.10]{librospringer2}}]
Let $\Omega= B(0,R)\assign \{q\in\mathbb H: |q|<R\}$ for some $R>0$. Then for any $f:\Omega\lto\mathbb H$ slice regular we have
\[ f(q)=\sum_{m=0}^{+\infty} q^m (m!)^{-1}f^{(m)}(0). \]
\end{theorem}

It can be defined a product between such functions which preserves slice regularity:
\begin{definition}
	Let $f(q)=\sum_{m=0}^{+\infty} q^m a_m$ and $g(q)=\sum_{m=0}^{+\infty} q^m b_m$ be two slice regular functions on a ball $B=B(0,R)$. Their \emph{$\ast$-product} $f\ast g$ is the slice regular function on $B$ given by \[ (f\ast g) (q)\assign \sum_{m=0}^{+\infty} q^m \left( \sum_{\ell=0}^{m} a_{m-\ell} b_{\ell}\right). \]
	This power series still converges in $B(0,R)$. 
	Since the product in $\mathbb H$ is associative, then also the $\ast$-product is associative; however, it is clearly noncommutative.
\end{definition}
In order to evaluate at a single point the $\ast$-product, we have the following formula.
\begin{proposition}[{\cite[p. 42]{librospringer2}}]\label{prop:formulastar}
	For $f,g$ as in the previous definition, it is
	\[ (f\ast g)(q)= \begin{cases}
		0 &\text{if }f(q)=0\\ f(q) g(f(q)^{-1} q f(q)), & \text{otherwise.}
	\end{cases} \]
\end{proposition}

The $\ast$-product, together with the standard sum, gives to $\mathcal{SR}(B(0,R))$ the structure of a (noncommutative) $\mathbb R$-algebra.

\begin{definition}
	A function $f:\Omega\lto \mathbb H$ is \emph{slice preserving} if for any $I\in\mathbb S$ it is $f(\Omega_I)\subseteq \mathbb C_I$.
	
	It is \emph{one slice preserving} if there exists a unique $I\in\mathbb S$ so that $f$ is $\mathbb C_I$-preserving: $f(\Omega_I)\subseteq \mathbb C_I$.
\end{definition}

\begin{proposition}\label{prop:slice_pres}Let $f\in\mathcal{SR}(B(0,R))$, $f(q)=\sum_{m=0}^{+\infty} q^m a_m$.
	\begin{enumeratealpha}
		\item $f$ is $\mathbb C_I$-preserving if and only if $a_m\in\mathbb C_I$ for all $m$.
		\item $f$ is slice preserving if and only if $a_m\in\mathbb R$ for all $m$.
	\end{enumeratealpha}
	
\end{proposition}
\begin{proof}\mbox{}
	
	\begin{enumeratealpha}
		\item\label{1list-a} If $a_m$ are all in $\mathbb C_I$, then clearly $f$ restricted to that slice has values in $\mathbb C_I$; on the other hand, for a slice preserving and regular function it is $f^{(m)}(q)\in\mathbb C_I$ for all $q$ in the $I$-slice, thus $a_m=(m!)^{-1}f^{(m)}(0)\in\mathbb C_I$.
		\item follows from \ref{1list-a}) since $\bigcap_{I\in\mathbb S}\mathbb C_I=\mathbb R$.\qedhere
	\end{enumeratealpha}
\end{proof}

If $f\in\mathcal{SR}(B(0,R))$ is a slice preserving regular function and $g$ is slice regular on the same ball, then $f\ast g=fg$ and also $f\ast g= g\ast f$ by simply using power series expansion and the fact that the coefficients for $f$ are real. Furthermore, if $f,g$ preserve the same slice $\mathbb C_I$, then it is $f\ast g=g\ast f$.

If $f(q)=\sum_{m=0}^{+\infty} q^m a_m\in\mathcal{SR}(B(0,R))$, we define its \emph{regular conjugate function} as $f^c(q)\assign \sum_{m=0}^{+\infty} q^m \overline a_m$ (of course, $f^c$ is still slice regular on $B(0,R)$) and its \emph{symmetrization} as $f^s\assign f\ast f^c=f^c \ast f$. One checks the coefficients of $f^s$ as a power series are real: indeed
\[ \left|\tmop{Im} \left( \sum_{\ell=0}^m a_{m-\ell} \overline a_{\ell}\right)\right|=\frac{1}{2}\left|\sum_{\ell=0}^m a_{m-\ell} \overline a_{\ell}- \sum_{\ell=0}^m a_{\ell}\overline a_{m-\ell}\right|=0. \] By Proposition~\ref{prop:slice_pres}, $f^s$ is {slice preserving} and acts as a complex holomorphic transformation on each slice. Hence $(f^s)^{-1}$, defined on the points of $B(0,R)$ which are not in the zero set $\mathcal Z(f^s)$ of $f^s$, is still a slice regular function \footnote{$B(0,R)\setminus \mathcal{Z}(f^s)$ is still a domain, as shown in \cite[Lemma 5.2]{librospringer2}}. So we can define \[ f^{-\ast}\assign (f^s)^{-1}\ast f^c=f^c\ast (f^s)^{-1}=(f^s)^{-1}f^c\] as a slice regular function on $B(0,R)\setminus \mathcal Z(f^s)$: this has the property that $f^{-\ast}\ast f=f\ast f^{-\ast}=\tmop{id}$. We call $f^{-\ast}$ the \emph{$\ast$-inverse} of $f$. For any $q\in B(0,R)\setminus \mathcal Z(f^s)$ it is $f^{-\ast}(q)=f^s(q)^{-1}f^c(q)=f(f^c(q)^{-1}q f^c(q))^{-1}$ by applying Proposition \ref{prop:formulastar} to $f^s=f^c\ast f$. Analogously, if $g$ is another slice regular function in $B(0,R)$, then the \emph{(left) regular quotient} $f^{-\ast}\ast g$ is defined and slice regular in $B(0,R)\setminus \mathcal Z(f^s)$ and here it evaluates as

\begin{equation}\label{eq:quotfg}
	(f^{-\ast}\ast g)(q)=f^{-\ast}(q)\cdot g(f^{c}(q)^{-1}f^s(q)qf^s(q)^{-1}f^{c}(q))=(f^{-1}g)(f^c(q)^{-1}q f^c(q)),
\end{equation}
since by slice preservation of $f^s$ it is $f^s(q)q f^s(q)^{-1}=q$.

One may also define a \emph{right regular quotient} $g\ast f^{-\ast}$, which is still slice regular in $B(0,R)\setminus \mathcal Z(f^s)$; in general it is different from $f^{-\ast}\ast g$, but they coincide when $f$ and $g$ $\ast$-commute.

\begin{remark}
	The map $\mathcal T_f: q\mapsto f(q)^{-1} q f(q)$, which appears in the evaluation formula for the $\ast$-product, is not a slice regular map. However, if we restrict it to $B(0,R)\setminus \mathcal Z(f^s)$, it is a real differentiable function which preserves the module and the real part: $|\mathcal{T}_f(q)|=|q|$ and $\tmop{Re} (\mathcal T_f(q))=\tmop{Re}q$. Moreover, the map $\mathcal{T}_{f^c}$ defined on $B(0,R)\setminus \mathcal Z(f^s)$ by $q\mapsto f^c(q)^{-1}q f^c(q)$ is its inverse: $\mathcal T_f \circ \mathcal T_{f^c}=\mathcal T_{f^c}\circ \mathcal T_f=\tmop{id}$ (see \cite[Prop. 5.32]{librospringer2}). So $\mathcal T_f$ is a diffeomorphism of $B(0,R)\setminus \mathcal Z(f^s)$ onto itself.
\end{remark}

\subsection{Regular linear fractional and M\"obius transformations}\mbox{}

In the context of slice regularity theory, linear fractional transformations and M\"obius transformation have been studied in \cite{volumeindam,moebius}, starting from considerations on their `classical' counterparts -- without the $\ast$-product -- already studied in \cite{poincare}. We recall here some fact we need.
\begin{definition}
	A \emph{regular linear fractional transformation} is a map of the form 
	\[ 
	\mathcal F_A: q\mapsto (qc+d)^{-\ast}\ast (qa+b)
	\] 
	associated to an invertible quaternionic matrix $A=\begin{psmallmatrix}
		a & c\\ b & d
	\end{psmallmatrix}\in \tmop{GL}(2,\mathbb H)$. A quaternionic $2\times 2$ matrix $A=\begin{psmallmatrix}
	a & c\\ b & d
	\end{psmallmatrix}$ is invertible if and only if its Dieudonn\'e determinant (see \cite{poincare}) is not zero: \[\tmop{det}_\mathbb H A\assign \sqrt{|a|^2|d|^2+|b|^2|c|^2-2\tmop{Re}(b\overline a c \overline d)}\neq 0.\]
	
	The map $\mathcal F_A$ belongs to the ring of quotients of $\mathcal{SR}(\mathbb H)$ with respect to the $\ast$-product; it is slice regular
	on $\mathbb H$ if $c=0$ and on $\mathbb H\setminus \mathbb S_{dc^{-1}}$ otherwise, where $\mathbb S_{dc^{-1}}\assign\{q:\tmop{Re}q=\tmop{Re}(dc^{-1}), |\tmop{Im}q|=|\tmop{Im}(dc^{-1})|\}=\mathcal{Z}((qc+d)^s)$. By \eqref{eq:quotfg}, one has $ \mathcal F_A=F_A\circ \mathcal T$, where
	\[
	F_A(q)=(qc+d)^{-1}(qa+b),\qquad \mathcal T=\mathcal T_{q\overline c+\overline d}:q\mapsto (q\overline c+\overline d)^{-1}\cdot q \cdot (q\overline c+\overline d).
	\]
	
\end{definition}

\begin{remark}\label{rmk:red}
	It is clear that for all $\lambda\in\mathbb R\setminus\{0\}$, it is $\mathcal F_{\lambda A}=\mathcal F_A$, hence we can consider $A$ belonging to a class in the quotient group $\tmop{PSL}(2,\mathbb H)=\tmop{GL}(2,\mathbb H)/\{\lambda \mathbf I:\lambda\in\mathbb R\setminus\{0\}\}$. In particular, we may choose $\lambda=(\det_\mathbb HA)^{-1/2}$ and so we can assume $A\in \tmop{SL}(2,\mathbb H)=\{A\in\tmop{GL}(2,\mathbb H):\det_\mathbb HA=1\}$. Be aware that the reduction to $\tmop{PSL}(2,\mathbb H)$ is not enough to have a bijective correspondence between matrices and transformations: for instance, $\mathcal F_{c\mathbf I}(q)=\mathcal F_\mathbf I=\tmop{id}$ for any $c\in\mathbb H\setminus\{0\}$.
\end{remark}

\begin{definition}
	A regular linear fractional transformation $\mathcal F_A$ such that it is slice regular on the quaternionic unit ball $\mathbb B\assign B(0,1)$ and such that $\mathcal F_A(\mathbb B)\subseteq \mathbb B$ is said a \emph{(regular) M\"obius transformation}. They are characterized by the following result.
\end{definition}

\begin{proposition}[{\cite[Cor. 7.2]{moebius}}]
	Let $\mathcal F_A$ be a regular linear fractional transformation associated to  some $A\in\tmop{GL}(2,\mathbb H)$.
	The following are equivalent:
	\begin{enumeratealpha}
		\item $\mathcal F_A $ is a regular M\"obius transformation.
		\item $(\det_\mathbb HA)^{-1/2}A\in \tmop{Sp}(1,1)=\{S\in \tmop{GL}(2,\mathbb H):\leftindex^t{\overline
			{S}}\cdot \begin{psmallmatrix} 1 & 0\\ 0 & -1\end{psmallmatrix} \cdot S=\begin{psmallmatrix} 1 & 0\\ 0 & -1\end{psmallmatrix} \}$.
		\item There exists (uniquely) $u\in\partial\mathbb B$ and $p\in\mathbb B$ such that 
		\[
		\mathcal F_A(q)=(1-q\overline p)^{-\ast}\ast(q-p)u=(\overline u\ast q-p \overline u)\ast(1- \overline p\ast q)^{-\ast}
		\]
	\end{enumeratealpha}
\end{proposition}
In particular, considering $u=1$ we get that for any $p\in\mathbb B$ we can associate a M\"obius transformation
\begin{equation}\label{mobius}
\mathcal M_p\assign (1-q\overline p)^{-\ast}\ast (q-p)=(q-p)\ast(1-q\overline p)^{-\ast}
\end{equation}
which is associated to an Hermitian matrix
\[
\frac{1}{\sqrt{1-|p|^2}}\begin{pmatrix}
	1 & -\overline p \\ -p & 1
\end{pmatrix}\in\tmop{Sp}(1,1).
\]
Evaluating $\mathcal M_p$ at any point $q$ in the unit ball gives $\mathcal M_p(q)=(M_p\circ \mathcal T_p)(q)$ where $M_p(q)$ is the `classical' M\"obius map $(1-q\overline p)^{-1}(q-p)$ and $\mathcal T_p(q)=(1-q p)^{-1} q (1-q p)$. Both these maps $M_p$ and $\mathcal T_p$ are bijective from $\mathbb B$ to itself, with inverse maps
\begin{align*}
M_p^{-1}=M_{-p} &:s\mapsto (1+s\overline p)^{-1}(s+p),\\
\mathcal T_p^{-1}=\mathcal T_{\overline{p}} &:s\mapsto (1-s\overline p)^{-1} s (1-s\overline p).
\end{align*}
Since right multiplication by a unimodular constant is an invertible operation too, we have proved that:
\begin{proposition}
All regular M\"obius transformations are bijective from $\mathbb B$ onto itself. 
\end{proposition}
Moreover, we remark that any point outside the unit ball in which a given regular M\"obius transformation is slice regular is mapped outside the unit ball.

A relevant fact of these maps is that -- if they are identified with the multiplicative group $\tmop{Sp}(1,1)$ -- they act on the space slice regular functions mapping the unit ball to itself $\mathcal{SR(\mathbb B,\mathbb B)}=\{f\in\mathcal{SR}(\mathbb B):f(\mathbb B)\subseteq\mathbb B\}$. 

\begin{proposition}[\cite{volumeindam,moebius}]\label{prop:action_sp}
 Let $\mathcal F_A$ be a regular M\"obius transformation and let us assume that the associated matrix is $A=(1-|p|^2)^{-1/2}\begin{psmallmatrix}
 	u & -\overline p  \\ -pu & 1
 \end{psmallmatrix}\in\tmop{Sp}(1,1)$ for some $u\in\partial\mathbb B$ and $p\in\mathbb B$, i.e., $\mathcal F_A=\mathcal M_p u$. If $f\in\mathcal{SR(\mathbb B,\mathbb B)}$, then the function
 \[
 f.A\assign (1- f\overline p)^{-\ast}\ast(f-p) u
 \] belongs to $\mathcal{SR}(\mathbb B,\mathbb B)$ and the map $\tmop{Sp}(1,1)\times\mathcal{SR(\mathbb B,\mathbb B)}\lto\mathcal{SR(\mathbb B,\mathbb B)},(A,f)\mapsto f.A$ is a \emph{right action}, i.e., $(f.A).B=f.(AB)$.
 
 Alternatively, we can define a left action of $\tmop{Sp}(1,1)$ over $\mathcal{SR(\mathbb B,\mathbb B)}$ by \[
 \leftindex^tA. f = (u\ast f-pu)\ast (1-\overline p \ast f)^{-\ast}.
 \]
 The two actions coincide when $A$ is Hermitian (i.e. $\overline{A}=\leftindex^tA$), or equivalently if $u=1$.
\end{proposition}

An easy consequence of what stated above is that any regular M\"obius transformation is bijective. 

We will use the following notation from now on:
\begin{definition}
If $A=(1-|p|^2)^{-1/2}\begin{psmallmatrix}
	1 & -\overline p  \\ -p & 1
\end{psmallmatrix}\in\tmop{Sp}(1,1)$ and $f$ is any function in $\mathcal{SR}(\mathbb B,\mathbb B)$, then we denote by $\mathcal M_p\bullet f$ the function
\[
\mathcal M_p\bullet f=f.A=\leftindex^tA.f=(1-f\overline p)^{-\ast}\ast(f-p)=(f-p)\ast(1-\overline p\ast f)^{-\ast}.
\]
\end{definition}
 Evaluating at a point $q\in\mathbb B$ such function gives
\[
(\mathcal M_p\bullet f)(q)=(1-f\overline p)^{-\ast}\ast(f-p)\mid_{q}=(1-f(\widetilde{\mathcal T}(q))\overline p)^{-1}(f(\widetilde{\mathcal T}(q))-p)=(M_p\circ f\circ \widetilde{\mathcal T})(q)
\]
where $\widetilde{\mathcal T}(q)=\varphi_f(q)^{-1}\cdot q\cdot \varphi_f(q)$, $\varphi_f(q)=(1-f\ast \overline p)^c (q)=1-(p\ast f^c)(q)$.

This notation is motivated by the fact that $\mathcal F_A$ coincides with $\mathcal M_p$ for $A$ given as above, and that represents, in some sense, a regular counterpart of the composition of a M\"obius map with a self-map $f$ of the unit ball. 
However, one should be careful: in general, we cannot replace $A$ with \emph{any} matrix $B$ such that $\mathcal F_B=\mathcal M_p$: we observed in Remark \ref{rmk:red} that the identity function, which coincides with $\mathcal M_0$, can be represented by any multiple $c\mathbf I$ of the identity matrix by a nonzero quaternion $c$, however it we choose $c=i$ and $f(q)=qj$, then it is easy to show that $f.\mathbf I(q)=f(q)=qj$ while $f.(i\mathbf I)(q)=i\ast qk=-qj$. Since for $p=0$ the matrix $A=\mathbf I$, for us it will be clear that $\mathcal M_0\bullet f=f$. 
\subsection{Regular Blaschke products}\mbox{}

We now introduce the analogous notion of a complex Blaschke product in the framework of quaternionic slice regularity, which already appeared in \cite{AlpayColomboSabadini13,AltavillaBisi19, weierstrass}.

\begin{definition}
	A \emph{(slice) regular (finite) Blaschke product} is a function $f$ over $\mathbb B$ which can be written in the form
	\begin{equation}\label{eq:blaschke}
		f=\mathcal M_{s_1}\ast \mathcal M_{s_2}\ast\ldots\ast\mathcal M_{s_\kappa}u,\qquad (\kappa\geq 1)
	\end{equation} where $u\in\partial \mathbb B$ is a unimodular constant and $s_1,\ldots,s_\kappa\in\mathbb B$. 
\end{definition}

	If $f$ is of this form, it clearly maps the unit ball to itself, it is slice regular, it extends continuously to the unit sphere $\partial \mathbb B$ which is $f$-invariant. 
	
	The zero set of this function, and in general of any slice regular function, can be described as a combination of (eventually repeated) $2$-spheres $\mathbb S_{p_1},\ldots,\mathbb{S}_{p_m}$ where $\mathbb S_{p_\ell}=\{q:\tmop{Re}q =\tmop{Re}(p_\ell), |q|=|p_\ell|\}$ and points $q_1,\ldots,q_n$, which could be not distinct or not disjoint from the spheres, such that the total multiplicity is $n+2m=\kappa$.  For these definitions and the general description of the zero sets of slice regular functions we refer to Chapter 3 of \cite{librospringer2}.
	
	This last remark says that, although the expression \eqref{eq:blaschke} is not unique, the number $\kappa$ is uniquely determined and it is called the \emph{degree} of the Blaschke product $f$. Blaschke products of degree $1$, $f=\mathcal M_{q_1} u$, are exactly the regular M\"obius transformations.

Now, the properties listed above give a complete characterization of slice regular Blaschke products.

\begin{lemma}
	\label{blaschke-c}Let $f : \mathbb{B} \lto \mathbb{B}$ be a
	slice regular function such that
	\begin{enumeratealpha}
		\item \label{blaschke-a}it extends continuously up to $\partial
		\mathbb{B}$ and $| f | \equiv 1$ on $\partial \mathbb{B}$;
		
		\item its set of zeros is given by a finite number of points $q_1, \ldots, q_n$ and
		spheres $\mathbb{S}_{p_1}, \ldots, \mathbb{S}_{p_m}$ (counted with their respective
		multiplicity), with $\kappa = n + 2 m$.
	\end{enumeratealpha}
	Then $f = B u$, where $B$ is a regular Blaschke product of degree $\kappa$ with the
	same zero set and $u$ is a unimodular quaternionic constant.
\end{lemma}

\begin{proof} 
	Let $B$ be a Blaschke product with zero set given by the points $q_1,\ldots, q_n$ and the spheres $\mathbb{S}_{p_1}, \ldots, \mathbb{S}_{p_m}$ with correspondent multiplicity. Let $f$ be as in the assumption. Then by a factorization argument one shows that both $f^{- \ast} \ast B$ and $B^{-\ast} \ast f$ are regular functions over $\mathbb{B}$, with continuous extension to the boundary. Moreover $| f^{- \ast} \ast B | = | B^{- \ast}\ast f |= 1$ on the whole $\partial \mathbb{B}$. Hence by the Maximum Modulus Principle {\cite[Thm 7.1]{librospringer2}} it is
	\[ 
	| B^{- \ast} \ast f | \leq 1, \qquad | f^{- \ast} \ast B | \leq 1 \qquad\text{on } \mathbb{B}. \]
	Let  $T_f (q) = f^c (q)^{- 1} q f^c (q)$ and $T_B (q)$ the same with $f$ replaced by $B$. Let $q_0$ be a point in the unit ball which is not a zero
	of $f^s$ (and thus of $B^s$), i.e., $q_0 \in \mathbb{B} \setminus \left[
	\left( \bigcup_{h = 1}^n \mathbb{S}_{q_h} \right) \cup \left( \bigcup_{\ell =
		1}^m \mathbb{S}_{p_\ell} \right) \right]$. \footnote{If $r\in\mathbb R$, then the 2-sphere $\mathbb S_r=\{r\}$ collapses to a single real point.} In this set the maps above are
	real diffeomorphisms (hence invertible). So we can define $\widetilde{q_0} =
	T_B^{- 1} (q_0)$, $\widehat{q_0} = T_f^{- 1} (q_0)$ and we have:
	\[ (B^{- \ast} \ast f) (\widetilde{q_0}) = B (q_0)^{- 1} f (q_0) \backassign
	u \qquad \Longrightarrow \qquad | u | \leq 1 ; \]
	\[ (f^{- \ast} \ast B) (\widehat{q_0}) = f (q_0)^{- 1} B (q_0) = u^{-1} \qquad
	\Longrightarrow \qquad | u | \geq 1. \]
	Therefore there exists a point $q_0 \in \mathbb{B}$ such that $B^{- \ast}
	\ast f$ evaluated at $q_0$ gives $u$, which has modulus one, hence by the
	Maximum Modulus Principle it is $B^{- \ast} \ast f \equiv u$, that is to say
	$f = B u$.
\end{proof}

\begin{proposition}\label{prop:blaschke_deg}
	Let $B$ be a regular Blaschke product of degree $\kappa$. If $\mathcal{M}_{q_0}$ is the
	regular M\"obius transformation associated to a point $q_0\in\mathbb B$, then $\mathcal{M}_{q_0} \bullet B$ is again a
	regular Blaschke product with the same degree. 
	
	Conversely, if $f\in\mathcal{SR}(\mathbb B)$ is such that $\mathcal{M}_{q_0} \bullet f$ is a Blaschke product of degree $\kappa$, then $f$ is a Blaschke product of degree $\kappa$.
\end{proposition}

\begin{proof} 
	We have that $\mathcal{M}_{q_0} \bullet B (q)
	= (B (q) - q_0) \ast (1 - \overline{q_0} \ast B (q))^{- \ast}$ is a slice
	regular function which clearly extends continuously to the boundary.
	Moreover it is not difficult to check that $| (\mathcal{M}_{q_0} \bullet B) (q) |
	\rightarrow 1$ as $q$ goes to any point of $\partial \mathbb{B}$. In order
	to conclude it remains to consider its zero set. So we need to consider the
	fiber $B^{- 1} (q_0)$.
	
	Let us assume $B =\mathcal{M}_{s_1} \ast \cdots \ast \mathcal{M}_{s_\kappa}$ (multiplying for a unimodular constant will not change the argument). The equation $B (q)
	= q_0$, where $q_0 \in \mathbb{B}$, will admit only solutions of modulus less than one, as otherwise there would be a M\"obius factor whose value in a point outside $\mathbb B$ is of norm less than one, and that is absurd.
	
	If $\kappa = 1$, it is clear that there is a unique solution of $B (q) = q_0$:
	\[ B(q)=q_0 \iff q-s_1 = (1- q\overline{s_1})q_0 \iff q = (s_1 + q_0) (1+\overline{s_1} q_0)^{-1}. \]
	If $\kappa = 2$, $B (q) = q_0$ can be rewritten as
	\[
	(1 - q \overline{s_1})^{-\ast}\ast(q - s_1) \ast (q - s_2)\ast (1 - q
	\overline{s_2})^{-\ast}= q_0,
	\]
	hence by left $\ast$-multiplying by $(1-q\overline{s_1})$ and right $\ast$-multiplying by $(1-q\overline{sp_2})$, then
	\[ (q - s_1) \ast (q - s_2) = (1 - q \overline{s_1}) q_0 \ast (1 - q
	\overline{s_2}),\]
	thus we get the regular polynomial equation
	\[ q^2 (1 - \overline{s_1} q_0  \overline{s_2}) + q (\overline{s_1} q_0 +
	q_0  \overline{s_2} - s_1 - s_2) + s_1 s_2 - q_0 = 0. \]
	The leading coefficient is not zero, so we reduce ourselves to consider the
	zero set of a slice regular polynomial of degree $2$, which has total
	multiplicity $2$.
	
	Now, when $\kappa \geq 3$, we need to be more careful. Indeed now some of the `star-denominators' are not the left or the right $\ast$-factor in the expression of $B(q)$, so we need to find a way to shift them. With this purpose, we observe that
	\[ (1 - q \alpha)^{- \ast} \ast (q - \beta ) = (q - \beta \gamma) \ast
	(\gamma - q \alpha)^{- \ast} \]
	where $\gamma = (1 - \alpha \beta)^{- 1}  (1 - \beta \alpha)\in\partial\mathbb B$. Indeed:
	\[ (1 - q \alpha) \ast (q - \beta \gamma) = q^2  (- \alpha) + q (1 + \alpha
	\beta \gamma) - \beta \gamma, \]
	\[ (q - \beta) \ast (\gamma - q \alpha) = q^2 (- \alpha) + q (\gamma + \beta
	\alpha) - \beta \gamma, \]
	\[ 1 + \alpha \beta \gamma = \gamma + \beta \alpha \Leftrightarrow (1 -
	\alpha \beta) \gamma = 1 - \beta \alpha . \]
	Let us remark that $| \gamma | = 1$. Using this trick and then proceeding as in the $\kappa=2$ case, one gets that the equation $B
	(q) = q_0$ is equivalent to $P_\kappa (q) = 0$, where $P_\kappa$ is a slice regular
	polynomial of degree $\kappa$, hence the set of its roots has total
	multiplicity $\kappa=n+2m$, where $m$ is the number of spherical zeros and $n$ the number of non-spherical zeros (counted with their respective multiplicity).
	
	For the converse part, it is enough to observe that if $\mathcal M_{q_0} \bullet f$ is a Blaschke product, then we can take $\mathcal M_{-q_0}$: as the associated matrices in $\tmop{Sp}(1,1)$ are such that their product is the identity matrix $\mathbb{I}$, then $\mathcal M_{-q_0}\bullet(\mathcal M_{q_0}\bullet f)=f$. Thus $f$ is a Blaschke product too.
\end{proof}


In the particular case $\kappa = 2$, we have also:

\begin{proposition}
	A finite Blaschke product $B$ is of degree two if and only if it exists a
	M\"obius transformation $\mathcal{M}_p$ and a point $q_c \in \mathbb{B}$
	such that $(\mathcal{M}_p \bullet B) u =\mathcal{M}_{q_c}^{\ast 2}
	=\mathcal{M}_{q_c} \ast \mathcal{M}_{q_c}$ for some $u \in \partial
	\mathbb{B}$.
\end{proposition}

\begin{proof} 
	The if part easily follows by the previous result.
	
	If we assume that $B =\mathcal{M}_{q_1} \ast \mathcal{M}_{q_2}$, without losing
	generality, then
	\begin{equation}
		\partial_C B (q) = (1 - q \overline{q_1})^{- \ast 2} \ast [(1 - | q_1 |^2)
		((q - q_2) \ast (1 - q \overline{q_2})) + (1 - | q_2 |^2) ((1 - q
		\overline{q_1}) \ast (q - q_1))] \ast (1 - q \overline{q_2})^{- \ast 2}
		\label{crit}
	\end{equation}
	The term within the square brackets in \eqref{crit} writes as a regular polynomial $q^2 a_2 + 2 q a_1 + a_0$, where
	\begin{eqnarray*}
		a_0 & = & - (1 - | q_1 |^2) q_2 - (1 - | q_2 |^2) q_1\\
		a_1 & = & 1 - | q_1 |^2  | q_2 |^2 \in \mathbb{R},\\
		a_2 & = & \overline{a_0} = - (1 - | q_1 |^2)  \overline{q_2} - (1 - | q_2
		|^2)  \overline{q_1} .
	\end{eqnarray*}
	When $q_1=-q_2$, this polynomial becomes $2 q (1-|q_1|^4)$ which has a unique root $q=0$. Otherwise, we have a polynomial of degree two for which $q = 0$ is not a root. A priori, we get that the set of zeros is non-empty in $\mathbb{H}$ (in particular, of total multiplicity two). Now it is clear that the coefficients above lie all in the same slice $\mathbb C_I$ for some $I\in\mathbb S$, so we may solve the equation as in the complex case and the roots will lie on the same slice $\mathbb C_I$ as the coefficients. These roots are also critical points of the restriction $B\mid_{\mathbb C_I}$ which is a complex Blaschke product by identifying $\mathbb{C}$ with $\mathbb C_I$. We know that in the complex case a Blaschke product of degree two admits a unique critical point inside the unit disk $\mathbb{D} \cong \mathbb{B} \cap \mathbb C_I$. We conclude that there exists $q_c \in \mathbb{B}$ such that $(\partial_C B) (q_c) = 0$. Let $p_c = B(q_c)$. Now by choosing $\mathcal{M}=\mathcal{M}_{p_c}$, then $\mathcal{M}_{p_c}\bullet B$ is a Blaschke product of degree two which vanishes at $q_c$. Now
	\begin{eqnarray*}
		\mathcal{M}_{p_c} \bullet B & = & (B - p_c) \ast (1 - \overline{p_c} \ast B)^{- \ast},\\
		\partial_C  (\mathcal{M}_{p_c} \bullet B) & = & \partial_C B \ast (1 - \overline{p_c}
		\ast B)^{- \ast} + (B - p_c) \ast \partial_C  ((1 - \overline{p_c} \ast B)^{-
			\ast}),
	\end{eqnarray*}
	hence $(\partial_C  (\mathcal{M}_{p_c} \bullet B)) (q_c) = 0$ as $B (q_c) - p_c = 0$ and $(\partial_C B) (q_c) = 0$. We get that the zero set of $\mathcal{M}_{p_c}\bullet B$ is given by $q_c$ counted twice, so it coincides with $\mathcal{M}_{q_c}^{\ast 2}$ up to right multiplication by a unimodular constant.
\end{proof}

We conclude this section by recalling that -- although we will not consider them in the following -- it is also possible to define regular \emph{infinite} Blaschke products. We refer to  \cite{AlpayColomboSabadini13,weierstrass} for their definition.

\section{The hyperbolic difference quotient and the three-points Schwarz--Pick Lemma}\label{sec:spl}
We start by recalling how to define a difference quotient in the context of the theory of slice regularity.
\begin{definition} [{\cite[Rmk. 1.6]{BSIndiana}}]
  Let $f: \mathbb{B} \lto \mathbb{B}$ be a slice regular
  function. Fix a point $p \in \mathbb{B}$. Then the {\emph{regular difference
  quotient}} of $f$ at $p$ is the function
  \[ \mathcal{R}_{f, p} (q) \assign (q - p)^{- \ast} \ast (f (q) - f (p)) . \]
  At a first look, since the term $(q - p)^{- \ast}$ is well defined only for $q \in \mathbb{B}$
  not in the sphere $\mathbb{S}_p = \{ q : q^2 - q \cdot 2 \tmop{Re} p + | p
  |^2 =0 \}$, the domain of $\mathcal{R}_{f,p}$ is only $\mathbb B\setminus \mathbb{S}_p$.  
  But the singularity at $\mathbb S_p$ is removable: indeed, by applying \cite[Prop. 3.18]{librospringer2} to
 $f-f(p)$, there exists a unique slice regular function $g:\mathbb B\lto \mathbb B$ for which
 $f(q)=f(p)+(q-p)\ast g(q)$, and it is clear that $g$ coincides with $\mathcal{R}_{f,p}$ in $\mathbb H\setminus \mathbb{S}_p$.
  The value in $p$ of this function
  coincides with the regular derivative $
  \partial_C f (p)$, whereas its value in $\overline{p}$ coincides by
  definition with the spherical derivative in $p$: $\partial_S f (p)\assign(2 \tmop{Im} p)^{- 1}  (f (p) - f (\overline{p}))$. \footnote{When $p$ is real, we will assume that the spherical derivative coincides with the Cullen derivative.}
\end{definition}

\begin{definition}
  Let $f : \mathbb{B} \lto \mathbb{B}$ be a slice regular
  function. Fix a point $p \in \mathbb{B}$. We define the \emph{hyperbolic difference quotient} at $p$
  \[ f^{\star}_p (q)\assign (1 - q \overline{p}) \ast
     \mathcal{R}_{f, p} \ast (1 - \overline{f (p)} \ast f (q))^{- \ast} . \]
  Then $f_p^{\star} : \mathbb{B} \lto \mathbb{H}$ is a slice regular
  function as a  $\ast$-product of regular functions: the last term is not singular since $f$ has values in the unit ball.
 In particular, its value in $q=p$ is called the \emph{hyperbolic derivative} of $f$ in $p$ and it is denoted by $f^h(p)$. By evaluating the $\ast$-products, it is related to the regular derivative by the formula
 \begin{equation}\label{eq:fhcullen}
 	f^h(p)= f_p^{\star} (p) = (1 - | p |^2) \partial_C f (p)  (1 - \overline{f (p)} f
 	(p_0))^{- 1},\quad\text{for some }p_0\in\mathbb S_p
 \end{equation}
 where $p_0\in\mathbb S_p$ depends on $p$ and on the function $f$ itself.
  Note that $f_p^{\star}$ can also be written as 
  \begin{equation*}
  {f_p^{\star} =\mathcal{M}_p^{-
  \ast} \ast (\mathcal{M}_{f (p)}  \bullet f)}
  \end{equation*}
  where $\mathcal{M}_p$ is the
  regular M\"obius transformation \eqref{mobius}
  and by $\mathcal{M}_{f(p)} \bullet f$ we mean the (right) action on $f$ of the matrix $\frac{1}{\sqrt{1-|f(p)|^2}}\begin{psmallmatrix}
  	1 & -\overline{f(p)} \\ -f(p) & 1
  \end{psmallmatrix}$ given by Proposition~\ref{prop:action_sp}, namely
  
\begin{equation}\label{action}
\mathcal{M}_{f(p)} \bullet f = (f - f(p)) \ast ( 1 -\overline{f(p)} \ast f )^{- \ast} = ( 1 - f \cdot \overline{f(p)} )^{- \ast} \ast (f - f(p)) . 
\end{equation}
\end{definition}
\subsection{\texorpdfstring{\bf The quaternionic Schwarz--Pick Lemma}{The quaternionic Schwarz--Pick Lemma}}\mbox{}

The following results have been obtained by Bisi and Stoppato in \cite{BSIndiana}.
\begin{proposition}[{\cite[Thm. 3.2]{BSIndiana}}]\label{prop:splzero}
$f:\mathbb B\lto \mathbb B$ be a slice regular function. Let $p\in\mathbb B$ such that $f(p)=0$. Then for all $q\in\mathbb B$ it is \begin{equation}\label{eq:spl_zero} |(\mathcal M_p^{-\ast}\ast f)(q)|\leq 1.\end{equation} The inequality is strict unless $f=\mathcal M_p u$ for some $u\in\partial\mathbb B$, i.e., $f$ is a M\"obius transformation vanishing at $p$.

\end{proposition}

\begin{theorem}[Quaternionic Schwarz--Pick Lemma {\cite[Thm. 3.7]{BSIndiana}}]\label{thm:spl}
 Let $f:\mathbb B\lto \mathbb B$ be a slice regular function. Let $p\in\mathbb B$. Then for all $q\in\mathbb B$
 \begin{equation}\label{eq:spl_quat}
  |(\mathcal{M}_{f(p)}\bullet f)(q)|\leq |\mathcal M_p(q)|,
 \end{equation}
 \begin{equation*}
  |\mathcal R_{f,p}\ast(1-\overline{f(p)}\ast f)^{-\ast}|_{\mid_q}\leq|(1-\overline{p}\ast q)^{-\ast}|.
 \end{equation*}
Moreover, for any $p\in\mathbb B$
\begin{equation*}
|\partial_Cf\ast(1-\overline{f(p)}\ast f)^{-\ast}|_{\mid_p}\leq\dfrac{1}{1-|p|^2},
\end{equation*}
\begin{equation*}
 \dfrac{|\partial_Sf(p)|}{|1-f^s(p)|}\leq \dfrac{1}{|1-\overline p^2|}.
\end{equation*}
If $f$ is a regular M\"obius transformation, i.e. of the form $\mathcal M_p u$ where $|u|=1$, then all formulas above are equalities. Otherwise, the inequality \eqref{eq:spl_quat} is strict for any $q\in\mathbb B\setminus\{p\}$ while the others are strict on the whole unit ball.

\end{theorem}
\begin{remark}
	Recall that, if we denote $M_p (q) = (1 - q \overline{p})^{- 1} (q - p)$, then
$\mathcal{M}_p (q)$ evaluates as $ M_p ((1 - q p)^{- 1} q (1 - q p))$. So the right hand side of \eqref{eq:spl_quat} can be interpreted, once $q \in \mathbb{B}$ is fixed, as the (classical)
pseudo-hyperbolic distance between $p$ and $\mathcal{T}_p (q) \assign (1 - q
p)^{- 1} q (1 - q p)$. On the other hand,
\[ (\mathcal{M}_{f(p)} \bullet f) (q) = (M_{f(p)}
\circ f) (\widetilde{\mathcal{T}}_{f, p}(q)), \]
where
\begin{equation*}\label{eq:ttilde}
\widetilde{\mathcal{T}}_{f,p} (q)\assign \phi(q)^{-1}\cdot q \cdot \phi(q),\quad \phi(q)= (1-f\ast \overline{f(p)})^c(q). 
\end{equation*}

The maps $\mathcal T_p$ and $\widetilde {\mathcal T}_{f,p}$ are not
regular maps, however they are real diffeomorphisms leaving the spheres of type $\mathbb S_p$ invariant, where defined: indeed it is straightforward to check that they preserve the real part and the modulus of any quaternion in the unit ball.

If $\rho_{\mathbb{B}}$ denotes the classical pseudo-hyperbolic
distance $\rho_{\mathbb{B}} (p, q) \assign | M_p (q) |$, then we have
\begin{equation}\label{eq3prho}
	\rho_{\mathbb{B}} (f(p), f (\widetilde{\mathcal{T}}_{f,p} (q))) \leq \rho_{\mathbb{B}} (p, \mathcal{T}_p (q)) . 
\end{equation}
The result is therefore not in general as in the complex case, where in the left hand side we just have that the pullback of the pseudo-hyperbolic distance with respect to an holomorphic function is less or equal than the pseudo-hyperbolic distance itself. Here the two maps $\mathcal T_p$ and $\widetilde{\mathcal T}_{f,p}$ appearing in the two sides are different.

We could also consider the {\emph{regular}} pseudo-hyperbolic distance
$\rho^{\ast}_{\mathbb{B}} (p, q) \assign | \mathcal{M}_p (q) |$. As stated in
{\cite{invariantmetrics}}, it is $\rho^{\ast}_{\mathbb{B}} (p, q) = \rho^{\ast}_{\mathbb{B}}
(q, p)$ and it is the distance associated to the reproducing kernel Hilbert
space $H^2 (\mathbb{B})$ with kernel $k_p (q) = (1 - q \overline{p})^{- \ast}$.
Then $\rho^{\ast}_{\mathbb{B}} (p, q)$ is exactly the right hand side of equation \eqref{spl}
and we have the relation
\[ \rho^{\ast}_{\mathbb{B}} (p, q) = \rho_{\mathbb{B}} (p, \mathcal{T}_p (q)) . \]

Now $\mathcal{T}_p$ is a real diffeomorphism on $\mathbb{B}$ with
inverse $\mathcal{T}_p^{-1}=\mathcal T_{\overline p}:q\mapsto (1-q\overline p)^{-1}q(1-q\overline p)$, so $\rho_{\mathbb{B}} (p, w) =
\rho^{\ast}_{\mathbb{B}} (p, \mathcal{T}_p^{-1} (w))=\rho^\ast_{\mathbb B}(p,\mathcal T_{\overline p}(w))$. Hence \eqref{eq:spl_quat} in terms of
$\rho^{\ast}_{\mathbb{B}}$ becomes
\[ \rho^{\ast}_{\mathbb{B}} \left( f(p), \mathcal{T}_{\overline{f(p)}} (f (\widetilde{\mathcal{T}}_{f,p} (q))) \right) \leq\rho^{\ast}_{\mathbb{B}} (p, q) . \]
In this way, it is clear that the use of the distance $\rho_\mathbb B^{\ast}$ does not give a simpler interpretation of Theorem~\ref{thm:3pspl}. Analogous considerations can be made for the inequality \eqref{eq:spl_quat} of Theorem \ref{thm:spl}.

In the following we will also use the \emph{Poincar\'e distance}:
\[
\delta_\mathbb B (p, q) \assign \tanh^{-1} \rho_\mathbb B(p,q)=\frac{1}{2}\log\frac{1+\rho_\mathbb B(p,q)}{1-\rho_\mathbb B(p,q)}.
\]
By monotonicity of the function $\tanh$, \eqref{eq3prho} still holds if we replace the pseudo-hyperbolic distance with the Poincar\'e distance:
\begin{equation}\label{eq3pdelta}
	\delta_{\mathbb{B}} (f(p), f (\widetilde{\mathcal{T}}_{f,p} (q))) \leq \delta_{\mathbb{B}} (p, \mathcal{T}_p (q)) . 
\end{equation}
This Poincar\'e metric has been first introduced in \cite{poincare} for the quaternionic setting. See also \cite{poincaretrends} for a comparison with the Kobayashi distance of $\mathbb B$ as a subset of $\mathbb C^2$.
\end{remark}

\subsection{\texorpdfstring{\bf Three-points quaternionic Schwarz--Pick Lemma}{Three-points quaternionic Schwarz--Pick Lemma}}\mbox{}

Now, let us return to the hyperbolic difference quotient $f^\star_p$: we can apply the previous results obtaining a Schwarz--Pick type statement which now involves three points of the unit ball.
\begin{theorem}
  Let $f:\mathbb B\lto \mathbb B$ be slice regular and $p\in\mathbb B$. For any $q \in \mathbb{B}$ it is $| f^{\star}_p (q) |
  \leq 1$. The strict inequality holds, except for the case when $f$ is a regular
  M\"obius transformation (times a unimodular constant). In the latter situation,
  it is $f^{\star}_p \equiv u \in \partial \mathbb{B}$.
\end{theorem}

\begin{proof} 
  The result follows by \eqref{eq:spl_zero} where $f$ is replaced by the function $g\assign\mathcal{M}_{f (p)} \bullet f$, which has a
  zero at the point $p$, so $\mathcal M_{g(p)}\bullet g=\mathcal M_0\bullet g=g$. 
\end{proof}

\begin{theorem}[\textbf{Three-points Schwarz--Pick Lemma}]\label{thm:3pspl}
Suppose $f : \mathbb{B} \lto \mathbb{B}$ is slice regular but not a
regular M\"obius transformation. Let $p,s\in\mathbb B$. Then for any $q\in\mathbb B$
\begin{equation*}
  | (\mathcal{M}_{f_p^{\star} (s)} \bullet f^{\star}_p) (q) | \leq |
  \mathcal{M}_s (q) | . \label{spl}
\end{equation*}

\end{theorem}

\begin{proof} 
Just apply \eqref{eq:spl_quat} to the function $f_p^{\star}$.
\end{proof}

To consider when equality holds in the theorem above, we observe the following:

\begin{proposition}
 Suppose $f : \mathbb{B} \lto \mathbb{B}$ is slice
  regular but not a regular M\"obius transformation and fix $p \in
  \mathbb{B}$. Then $f^{\star}_p$ is a regular M\"obius transformation (up to
  right multiplication by unimodular constants) if and only if $f$ is a
  regular Blaschke product of degree $2$.
\end{proposition}

\begin{proof} 
  By the very definition of $f_p^{\star}$, it is
  \[ \mathcal{M}_p \ast f_p^{\star} =\mathcal{M}_{f (p)} \bullet f. \]
  Then it is clear that if $f_p^{\star}$ is a regular M\"obius
  transformation, then $\mathcal{M}_{f (p)} \bullet f$ is a  regular Blaschke product
  of degree two. Hence $f$ is also a regular Blaschke product of degree $2$. On the
  other side, if $f$ is a regular Blaschke product of degree $2$, then the right hand side
  is a regular Blaschke product vanishing at $p$, hence it writes down as
  $\mathcal{M}_p \ast \mathcal{M}_{p'} u$ for some $p' \in \mathbb{B}$,
  implying that $f_p^{\star} =\mathcal{M}_{p'} u$.
\end{proof}

As a direct consequence one gets that:
\begin{theorem}
  Equality holds in \eqref{spl} for some $q\in\mathbb B, q\neq s$ if and only if the function $f$ is a regular
  Blaschke product of degree $2$. In this case, equality holds for all $q\in\mathbb B$.
\end{theorem}

\begin{remark}
	In the proof of Theorem \ref{thm:spl} given in \cite{BSIndiana}, it is claimed that $\partial_S  (\mathcal{M}_{f (p)} \bullet f) (p) = (1 -\overline{f^s (p)})^{- 1} \partial_S f (p)$. From this we can get a formula for $f^\star_{p}$ evaluated in $\overline p$:
	\begin{equation}\label{eq:fstar_dspher}
		f_p^{\star} (\overline{p}) = (1 - \overline{p}^2) \mathcal{R}_{(\mathcal{M}_{f (p)}
			\bullet f), p} (\overline{p}) = (1 - \overline{p}^2)  (1 - \overline{f^s (p)})^{- 1}
		\partial_S f (p) .
	\end{equation}
\end{remark}

\section{Some consequences. Estimates in the case of \texorpdfstring{$f(0)=0$}{f(0)=0}} \label{sec:conseq}

\begin{lemma}\label{lem:sw}
 Let $f:\mathbb B \lto \mathbb B$ be slice regular and such that $f(0)=0$. Then for any $q_0\in\mathbb B$ it is
 \begin{equation*}
  f^{\star}_0(q_0)=f^{\star}_{q_0}(0)= q_0^{-1} f(q_0).
 \end{equation*}
\end{lemma}
\begin{proof} 
 Since $\mathcal M_0$ is the identity function $q\mapsto q$ and $\mathcal M_0\bullet f=f$, it is clear that $f^{\star}_0(q_0)=(\mathcal M_0^{-\ast}\ast (\mathcal M_0\bullet f))(q_0)=q_0^{-1} f(q_0)$. On the other side, since $0$ is invariant by the rotations appearing by evaluating the $\ast$-products punctually, then
\begin{equation*}
 f^{\star}_{q_0}(0)=(\mathcal M_{q_0}^{-\ast}\ast(\mathcal M_{f(q_0)}\bullet f))(0)=M_{q_0}(0)^{-1}M_{f(q_0)}(0)=(-q_0)^{-1}(-f(q_0))=q_0^{-1} f(q_0). \qedhere
\end{equation*}
\end{proof}

More generally, we have the following
\begin{lemma}\label{lem:swm}
Let $f:\mathbb B \lto \mathbb B$ be slice regular. Whenever $r,s\in(-1,1)$ are real, it is $\left|f^{\star}_r (s)\right|=\left|f^{\star}_s(r)\right|$.
\end{lemma}

\begin{proof} 
Using the fact that $\mathcal M_r=M_r$ is a slice preserving function and that $\mathcal M_{f(r)}\bullet f$ coincides with $M_{f(r)}\circ f$ on the real interval $(-1,1)$, one gets
 \[
  f^{\star}_r(s)=(\mathcal{M}_r^{-\ast}\ast(\mathcal{M}_{f(r)}\bullet f))(s)=M_r(s)^{-1}M_{f(r)}(f(s)).
 \]
 Analogously,
 \[
  f^{\star}_s(r)=(\mathcal{M}_s^{-\ast}\ast(\mathcal{M}_{f(s)}\bullet f))(r)=M_s(r)^{-1}M_{f(s)}(f(r)).
 \]
Now one concludes by recalling that $|M_p(q)|=|M_q(p)|$ for any $q,p\in\mathbb B$, as
\[
 M_p(q)= - (1-q\overline p)^{-1}(1-p\overline q) M_q (p)
\]
and the factor $- (1-q\overline p)^{-1}(1-p\overline q)$ has clearly modulus one.
\end{proof}

 Lemma \ref{lem:swm} is important to get the following `four-point' result (see \cite[Cor 3.3]{BeardonMinda04}), where we have to assume that two of the points are real. Here we consider the Poincar\'e distance $\delta_\mathbb B$.

\begin{corollary}
 Let $r,s\in (-1,1)$ and $q,p\in\mathbb B$. For $f:\mathbb B\lto\mathbb B$ slice regular function which is not a regular M\"obius transformation, we have
 \[
\delta_{\mathbb{B}}(0, f^{\star}_r(q))\leq \delta_{\mathbb B}(0,f^{\star}_s(p))+\delta_{\mathbb{B}}(p,r)+\delta_{\mathbb{B}}(q,s).
 \]
\end{corollary}

\begin{proof} 
By the triangle inequality and \eqref{eq3pdelta}, where we replace $f$ with $f_r^\ast$ and the two generic points $p$ and $q$ appearing there with $0$ and $s$, it is
 \begin{equation*}
\delta_{\mathbb{B}}(0, f^{\star}_r(q))\leq \delta_{\mathbb B}(0, f^{\star}_r(s))+\delta_{\mathbb B}(f^{\star}_r(q),f^{\star}_r(s))\leq \delta_{\mathbb B}(0,f^{\star}_r(s))+\delta_{\mathbb B}(q,s).
 \end{equation*}
By Lemma \ref{lem:swm} $|f^{\star}_r(s)|=|f^{\star}_s(r)|$, so their Poincar\'e distance from the origin is the same. Thus
\[
 \delta_{\mathbb{B}}(0, f^{\star}_r(q))\leq \delta_{\mathbb B}(0,f^{\star}_s(r))+\delta_{\mathbb B}(q,s).
\]
Applying, as above, triangle inequality and \eqref{eq3pdelta} to the first term on the right hand side, one gets
\[
 \delta_{\mathbb{B}}(0, f^{\star}_r(q))\leq \delta_{\mathbb B}(0,f^{\star}_s(p))+\delta_{\mathbb B}(p,r)+\delta_{\mathbb B}(q,s).\qedhere
\]
\end{proof}

\subsection{Distortion estimates in terms of the hyperbolic derivative}
{Almost like the complex case \cite{BeardonMinda04}, one can use the three-point Schwarz--Pick Lemma to get bounds for the values of the hyperbolic derivative of a regular self-map of the unit ball fixing the origin. As the hyperbolic derivative is related to the Cullen derivative, this estimates can be seen as distortion estimates, giving control to how the map can deform lengths and areas along slices going towards the boundary.
First, we need to show how Euclidean and pseudo-hyperbolic balls are related. Indeed, they are the same but with different centers and radii:}

\begin{lemma}\label{lem:balls}
 Let $c_0\in\mathbb B$ and $0<r_0<1$. Let us consider the pseudo-hyperbolic ball $B_\rho(c_0,r_0)\assign\{q\in\mathbb B : \rho_\mathbb B(q,c_0)<r_0\}$. This coincides with the Euclidean ball $B(c_1,r_1)\assign\{q\in\mathbb H:|q-c_1|<r_1\}$ where
 \[
 \begin{cases}
c_1 = \dfrac{1-r_0^2}{1-|c_0|^2 r_0^2} c_0,\\
r_1 = \dfrac{1-|c_0|^2}{1-|c_0|^2 r_0^2} r_0.
 \end{cases}
\]

\end{lemma}

\begin{proof} 
 By observing that the pseudo-hyperbolic distance is invariant by multiplication with a unimodular constant, we may assume that $c_0=|c_0|\in [0,1)$ and then deduce the general case by applying the inverse rotation. With this assumption, for any $I\in\mathbb S$ we observe that $B_\rho (c_0,r_0)\cap \mathbb B_I$ is an open disk which is symmetric with respect to the real line, described by the equation\footnote{Symmetry w.r.t. the real line follows by invariance with respect to conjugation $q\mapsto \overline q$} $|q-c_0|<r_0 |1-c_0 q|$ . Hence, to find the (Euclidean) center and radius of such disk, we compute the intersections of the boundary with the real line:
 \begin{equation}\label{eq:intersectionball}
  x-c_0 = \pm r_0 (1- c_0 x) \implies x_{\pm}=\dfrac{c_0\pm r_0}{1 \pm c_0 r_0}.
 \end{equation}
Hence $B_\rho (c_0,r_0)\cap \mathbb B_I$ corresponds to a disk along the slice of center
\[
 c_1=\frac{1}{2}(x_++x_-)=\frac{1}{2}\dfrac{(c_0+r_0)(1-c_0 r_0)+(c_0-r_0)(1+c_0r_0)}{1-c_0^2 r_0^2}=\dfrac{c_0 (1-r_0^2)}{1-c_0^2 r_0^2}
\]
and radius
\[
 r_1=\frac{1}{2}(x_+-x_-)=\frac{1}{2}\dfrac{(c_0+r_0)(1-c_0 r_0)-(c_0-r_0)(1+c_0r_0)}{1-c_0^2 r_0^2}=\dfrac{r_0 (1-c_0^2)}{1-c_0^2 r_0^2}.
\]
So $B_\rho (c_0,r_0)\cap\mathbb B_I=B(c_1,r_1)\cap\mathbb B_I=\{x+Iy\in\mathbb B: |x+Iy -c_1|<r_1\}$ for all $I\in\mathbb S$.
\end{proof}

{Now we can show the following inequality:}
\begin{proposition}[\textbf{Quaternionic Dieudonn\'e's estimate 1}]\label{prop:died1}
Let $f:\mathbb B \lto \mathbb B$ be slice regular, such that $f(0)=0$. Suppose $q_0\in\mathbb B$ and set $\alpha_{f}(q_0)\assign \frac{1-|f(q_0)|^2}{1-|q_0|^2}$. Then
\begin{equation}\label{dieud_1}
 \left| \alpha_f(q_0) f^h(q_0) - q_0^{-1} f(q_0) \right|\leq \dfrac{|q_0|^2-|f(q_0)|^2}{|q_0|(1-|q_0|^2)}.
\end{equation}
\end{proposition}
\begin{proof} 
 By \eqref{eq3prho} applied to the function $f^\star_{q_0}$, we have $\rho_{\mathbb B}(f^h(q_0),f^{\star}_{q_0}(0))\leq \rho_\mathbb B(q_0,0)=|q_0|$: so by Lemma \ref{lem:sw} $f^h(q_0)$ lies in the pseudo-hyperbolic ball $B_\rho (q_0^{-1} f(q_0), |q_0|)$. By Lemma \ref{lem:balls}
 \[
  B_\rho (c_0,r_0)=B\left(c_0\frac{1-r_0^2}{1-|c_0|^2 r_0^2},r_0\frac{1-|c_0|^2}{1-|c_0|^2 r_0^2}\right),
 \]
so one gets that $f^h(q_0)\in B\left(\alpha_f(q_0)^{-1} q_0^{-1} f(q_0), \frac{|q_0|^2-|f(q_0)|^2}{|q_0|(1-|f(q_0)|^2)}\right)$, which is equivalent to \eqref{dieud_1}.

\end{proof}

{A direct consequence is a distortion estimate at a point $q_0$, depending on the norm $|q_0|$, supposing that we are close enough to the boundary.}
\begin{corollary}[\textbf{Quaternionic Dieudonn\'e's estimate 2}]
 By the same assumptions of Proposition \ref{prop:died1}, \begin{equation}\label{dieud_2}
 \left|f^h(q_0)\right|\leq \begin{cases} \alpha_f(q_0)^{-1} & \text{if}\quad|q_0|\leq \sqrt 2-1\\ \alpha_f(q_0)^{-1}\dfrac
 {(1+|q_0|^2)^2}{4|q_0|(1-|q_0|^2)} & \text{if}\quad\sqrt 2-1<|q_0|< 1.\end{cases}
\end{equation}
\end{corollary}
\begin{proof} 
 Inequalities \eqref{dieud_2} can be obtained from \eqref{dieud_1} by using exactly the same argument shown in \cite[pag. 199-200]{DurenUnivalent}.
\end{proof}

{Another estimate for hyperbolic derivative can be obtained, involving the Cullen derivative at the origin and the norm of the point:}
\begin{proposition}[\textbf{Quaternionic Goluzin's estimate}]
 Let $f:\mathbb B \lto \mathbb B$ be a slice regular self-map of the unit ball such that $f(0)=0$. Then for any $q_0\in\mathbb B$ it is
 \begin{equation}
  \label{goluzin}
  \left|f^h(q_0)\right|\leq \dfrac{\left|\partial_C f(0)\right|+\frac{2|q_0|}{1+|q_0|^2}}{1+\left| \partial_C f(0)\right|\frac{2|q_0|}{1+|q_0|^2}}.
 \end{equation}

\end{proposition}

\begin{proof} 
First we observe that\[
f^h(0)=f^{\star}_0(0)=(\mathcal M_{0}^{-\ast}\ast(\mathcal M_0\bullet f))(0)= q^{-1} f(q)\mid_{q=0}=\partial_C f(0).\]
 We do our computation with the Poincar\'e distance: by the triangle inequality and \eqref{eq3pdelta} we have
 \[
  \delta_\mathbb B(0, f^h(q_0))\leq \delta_\mathbb B (0, f_{q_0}^{\star} (0))+\delta_\mathbb B (q_0,0).
 \]
 Now the first term in the right hand side coincides with $\delta_\mathbb B (0, f_0^{\star}(q_0))$ by Lemma \ref{lem:sw}. By the same argument as above
 \[
  \delta_\mathbb B(0, f^{\star}_0(q_0))\leq \delta_\mathbb B (0, f^h(0))+\delta_\mathbb B (q_0,0).
 \]
In the end we get $ \delta_\mathbb B(0, f^h(q_0))-\delta_\mathbb B (0, f^h(0))\leq 2 \delta_\mathbb B(q_0,0)$. In all the computations made above we can exchange $0$ and $q_0$, thus
\[
\left|\delta_\mathbb B(0, f^h(q_0))-\delta_\mathbb B (0, f^h(0))\right|\leq 2 \delta_\mathbb B(q_0,0)=2 \tanh^{-1}|q_0|,
\] where the left hand side equals $\delta_\mathbb B (|f^h(q_0)|,|\partial_C f(0)|)$. Applying the function $\tanh$ to both sides and recalling the duplication formula $\tanh (2x)=\tfrac{2 \tanh x}{1+ \tanh^2 x}$, we obtain that
\begin{align*}
 \dfrac{|f^h(q_0)|-|\partial_C f(0)|}{1-|\partial_C f(0)||f^h(q_0)|} &\leq \rho_\mathbb B(|f^h(q_0)|,|\partial_C f(0)|)\leq \dfrac{2|q_0|}{1+|q_0|^2}\\
 \implies|f^h(q_0)|-|\partial_C f(0)| & \leq \dfrac{2|q_0|}{1+|q_0|^2}-\left(|\partial_C f(0)|\dfrac{2|q_0|}{1+|q_0|^2}\right)|f^h(q_0)|.
\end{align*}
By solving the last inequality in $|f^h(q_0)|$ one gets \eqref{goluzin}.\qedhere

\end{proof}

\subsection{\texorpdfstring{\bf Distortion estimates for the class $\boldsymbol{\mathcal B_\alpha}$}{Distortion estimates for the class Bα}}\mbox{}

Let $\alpha\in[0,1)$ and let $\mathcal B_\alpha$ be the class of slice regular functions $f:\mathbb B \lto \mathbb B$ such that $f(0)=0$ and $\partial_C f (0)=\alpha$.  Let us observe that in this case the Cullen derivative at zero coincides with the hyperbolic derivative $f^h(0)$ by the definition of the latter. For such functions we have the following estimates.

\begin{proposition}[see {\cite[Thm. 6.1]{BeardonMinda04}}]
 Let $0\leq\alpha<1$ and $f\in\mathcal B_\alpha$. Then for all $q\in\mathbb B$
 \begin{equation*}
 \dfrac{|f^h(q)-\alpha|}{|1-\alpha f^h(q)|}\leq\dfrac{2|q|}{1+|q|^2}
 \end{equation*}
 and moreover
 \begin{equation}\label{eq:bound_balpha}
  \dfrac{\alpha |q|^2-2|q|+\alpha}{|q|^2-2\alpha |q|+1}\leq \tmop{Re} f^h (q) \leq |f^h(q)|\leq \dfrac{\alpha |q|^2+2|q|+\alpha}{|q|^2+2\alpha |q|+1}.
 \end{equation}
 The same estimates hold if we replace $f^h(q)$ with $f^{\star}_q(\overline{q})$.
\end{proposition}
\begin{proof} 
 If $f\in\mathcal B_\alpha$, then it is not a M\"obius transformation: otherwise $f(0)=0$ would imply that $f(q)=q\cdot u$, $u\in\partial \mathbb B$ and then it would be $\alpha=\partial_C f(0)=u$, a contradiction. 

 By applying the triangle inequality, recalling that by Lemma \ref{lem:sw} $f^{\star}_0(q)=f^{\star}_q(0)$ and then using \eqref{eq3pdelta} twice one gets that
 \begin{align*}
  \delta_\mathbb B (\alpha, f^h(q)) &=\delta_\mathbb B (f^{\star}_0(0),f^{\star}_q(q))=\delta_\mathbb B (f^{\star}_0(0), f^{\star}_0(q))+\delta_\mathbb B (f^{\star}_0(q), f^{\star}_q(q))\\
  &=\delta_\mathbb B (f^{\star}_0(0), f^{\star}_0(q))+\delta_\mathbb B (f^{\star}_q(0), f^{\star}_q(q))\leq 2 \delta_\mathbb B(0,q),
 \end{align*}
which can be rewritten as
\begin{equation*}
\tanh^{-1} \rho_\mathbb B(\alpha,f^h(q))\leq 2 \tanh^{-1}\rho_\mathbb B(0,q)=2 \tanh^{-1}|q|;
\end{equation*}
henceforth, by monotonicity of $\tanh$ and its duplication formula we have that
\begin{equation}\label{eq:distortion1}
 \rho_\mathbb B (\alpha, f^h(q))\leq \frac{2|q|}{1+|q|^2}
\end{equation}
and the left hand side corresponds by definition to $|M_\alpha(f^h(q))|=|1-\alpha f^h(q)|^{-1}|f^h(q)-\alpha|$.

Now, by Lemma \ref{lem:balls} the inequality \eqref{eq:distortion1} says that for any $q\in\mathbb B$, $f^h(q)$ lies on a Euclidean ball contained in $\mathbb B$ which is symmetric with respect to $\mathbb R$ since $\alpha$ is real: the intersections with the real line of such ball are given by \eqref{eq:intersectionball} where $c_0=\alpha$ and $r_0=2|q|/(1+|q|^2)$:
\[
 x_{\pm}=\dfrac{\alpha\pm 2|q|/(1+|q|^2)}{1\pm 2\alpha|q|/(1+|q|^2)}= \dfrac{\alpha|q|^2\pm 2|q|+\alpha}{|q|^2\pm 2\alpha|q|+1},\qquad x_- < x_{+},
\]
and its (Euclidean) center is
\[
\dfrac{x_++x_-}{2}=\left[ \dfrac{1-\left(\frac{2|q|}{1+|q|^2}\right)^2}{1-\alpha^2\left(\frac{2|q|}{1+|q|^2}\right)^2}\right]\cdot \alpha=\left[\dfrac{(1-|q|^2)^2}{1+2(1-\alpha^2)|q|^2+|q|^4}\right]\cdot \alpha \geq 0.
\]

Hence the real part of $f^h(q)$ lies between $x_{-}$ and $x_{+}$. Since $x_{+}\geq \max\{0,-x_{-}\}$, $x_+$ gives an upper bound also for $|f^h(q)|$.

If now we consider $f^{\star}_q(\overline{q})$, we have
 \begin{align*}
  \delta_\mathbb B (\alpha, f^{\star}_q(\overline{q})) &=\delta_\mathbb B (f^{\star}_0(0),f^{\star}_q(\overline{q}))=\delta_\mathbb B (f^{\star}_0(0), f^{\star}_0(q))+\delta_\mathbb B (f^{\star}_0(q), f^{\star}_q(\overline{q}))\\
  &=\delta_\mathbb B (f^{\star}_0(0), f^{\star}_0(q))+\delta_\mathbb B (f^{\star}_q(0), f^{\star}_q(\overline{q}))\leq \delta_\mathbb B(0,q)+\delta_\mathbb B (0,\overline{q})=2 \delta_\mathbb B(0,q)
 \end{align*}
 and so we get the same estimates by an identical argument.
\end{proof}
Given $\alpha\in[0,1)$, as above, let $\varrho(\alpha):=\dfrac{\alpha}{1+\sqrt{1-\alpha^2}}$. Clearly $\varrho(\alpha)<1$ and $\varrho(0)=0$.
\begin{corollary}[see {\cite[Cor. 6.2]{BeardonMinda04}}]
If $f\in\mathcal B_\alpha$, then for $|q|<\varrho(\alpha)$ we have that $\tmop{Re} f^h(q)>0$ and $\tmop{Re} f^{\star}_q(\overline q)>0$. In particular, both the Cullen and the spherical derivatives at $q$, $\partial_C f(q)$ and $\partial_S f(q)$, don't vanish.
\end{corollary}
\begin{proof}
	Observe that the left hand side of \eqref{eq:bound_balpha} is positive when $|q|<\varrho(\alpha)$ and vanishes for $|q|=\varrho(\alpha)$. In particular, we get that $f^h(q)$ and $f^\ast_q(\overline q)$ are different from zero and so are $\partial_Cf(q)$ and $\partial_Sf(q)$ by \eqref{eq:fhcullen} and \eqref{eq:fstar_dspher}.
\end{proof}
\begin{remark}
	In the complex setting, the estimate above can be used to conclude that the radius of univalence of the analogous class $\mathcal B_\alpha$ (\cite[Cor. 6.2]{BeardonMinda04}) is exactly $\varrho(\alpha)$. Indeed, the fact that the derivative at $q$ is not zero implies that we have (local) injectivity by the Inverse Function Theorem.
	
	In the quaternionic framework, however, the nonvanishing of both the Cullen and the spherical derivative at $q$ is not sufficient to conclude that the differential at $q$ is of maximal rank. We recall from \cite{ghiloniperotti21} that the differential of a slice regular function $f$ at a point $q=x+yI\notin \mathbb R$ computed on a tangent vector $\alpha$ is
	\[
	(df)_q(\alpha)=\pi_I(\alpha)\partial_Cf(q)+\pi_I^\perp(\alpha)\partial_Sf(q),
	\]
	where $\pi_I$ is the projection on the slice $\mathbb C_I$ and $\pi_I^\perp$ the projection on its orthogonal complement $\mathbb C_I^\perp$. Moreover, the determinant of the Jacobian matrix is given by \[
	\det J_f(q)=\big|\pi_I\big(\partial_Cf(q)\overline{\partial_S f(q)}\big)\big|^2.
	\]
	Hence the rank of $df$ at $q$ is maximal exactly when $\partial_S f(q)\neq 0$ and $\partial_C f(q)(\partial_S f(q))^{-1}\notin \mathbb C_I^\perp$, which is more than just requiring $\partial_Cf(q)\neq 0$ and $\partial_S f(q)\neq 0$. So an extra argument is required if one wants to conclude that $\varrho(\alpha)$ is the univalence radius of $\mathcal B_\alpha$ as a consequence of the estimate \eqref{eq:bound_balpha}.
	However, the fact that $\varrho(\alpha)$ is such a radius has already been proven in \cite{BSLandau}, but with a different approach.
\end{remark}

%
%
%

\section{Iterated hyperbolic difference quotients and the multipoint Schwarz--Pick Lemma}\label{sec:iterated}
 We have seen that if $f:\mathbb B \lto \mathbb B$ is slice regular and not a M\"obius transformation, then for any $p_1\in\mathbb B$ the function $f^{\star}_{p_1}$ is a slice regular self-map of the unit ball. We can define the hyperbolic difference quotient of $f^{\star}_{p_1}$ with respect to $p_2\in\mathbb B$, getting a unimodular constant whether $f^{\star}_{p_1}$ is M\"obius -- hence $f$ is a regular Blaschke product of degree two -- or obtaining a slice regular self-map of $\mathbb B$ otherwise. In the latter case we can go further with the iteration of such procedure taking the hyperbolic difference quotient with respect to a point $p_3$. In the end we have the following definition (see \cite{ChoKimSugawa12} or \cite{BaribeauRivardWegert} for the complex case).
\begin{definition}

 Let $f:\mathbb B\lto \mathbb B$ be a slice regular function, let $n\in\mathbb N,n>0$ and $p_1,\ldots, p_n$ points in $\mathbb B$. Then we set $f^{\{1\}}_{p_1}(q)\assign f^{\star}_{p_1}(q)$ and for $n>1$
 
 \begin{equation*}
f^{\{n\}}_{p_1,\ldots,p_n}(q)\assign\begin{cases} f^{\{n-1\}}_{p_1,\ldots,p_{n-1}} \quad\text{if }f^{\{n-1\}}_{p_1,\ldots,p_{n-1}} \text{ is a unimodular constant,}\\
(\mathcal{M}_{p_n}^{-\ast}\ast(\mathcal{M}_{f^{\{n-1\}}_{p_1,\ldots,p_{n-1}}(p_n)}\bullet f^{\{n-1\}}_{p_1,\ldots,p_{n-1}}))(q) \quad\text{ otherwise.}
                                \end{cases}
 \end{equation*}
 Let us observe that $f^{\{n\}}_{p_1,\ldots,p_n}=(f^{\{n-1\}}_{p_1,\ldots,p_{n-1}})^{\star}_{p_n}$ in the second case; moreover we get a unimodular constant in the case $f$ is a regular Blaschke product of degree less or equal to $n$, a M\"obius transformation if $f$ is Blaschke of degree $n+1$ and a regular Blaschke product of degree $\kappa\geq 2$ if $f$ is Blaschke of degree $n+\kappa$.
\end{definition}
\begin{theorem}
 Let $f:\mathbb B\lto \mathbb B$ be a slice regular function, $n>0$ integer and $p_1,\ldots, p_{n}\in\mathbb B$. Then
 \begin{equation*}
  |f^{\{n\}}_{p_1,\ldots,p_{n}}(q)|\leq 1\qquad\text{for all }q\in\mathbb B
 \end{equation*}
with strict inequality holding for each point, unless $f$ is a regular Blaschke product of degree $\leq n$, as the left term is a unimodular constant independent of $q$.
\end{theorem}
\begin{proof} 
 It is Proposition \ref{prop:splzero} when $n=1$. For $n>1$, just apply the same result to the function $f^{\{n-1\}}_{p_1,\ldots,p_{n-1}}$.
\end{proof}

\begin{theorem}[\textbf{Multipoint Schwarz--Pick Lemma}]
 Let $f:\mathbb B\lto \mathbb B$ be a slice regular function, $n>0$ integer and $p,p_1,\ldots, p_{n}\in\mathbb B$. If $f$ is not a regular Blaschke product of degree $\leq n$, then
 \[
|(\mathcal M_{f^{\{n\}}_{p_1,\ldots,p_n}(p)}\bullet f^{\{n\}}_{p_1,\ldots,p_n})(q)|\leq |\mathcal{M}_p(q)|
 \]
 for all $q\in\mathbb B$ and the inequality is strict for $q\in\mathbb B\setminus\{p\}$ unless $f$ is a regular Blaschke product of degree $n+1$.
\end{theorem}
\begin{proof} 
By Theorem \ref{thm:spl} applied to $f^{\{n\}}_{p_1,\ldots,p_n}$, which is a slice regular self-map of $\mathbb B$ by assumption.
\end{proof}
We can rewrite this result in terms of pseudo-hyperbolic metric:
\begin{equation*}
\rho_{\mathbb{B}} \left(f^{\{n\}}_{p_1,\ldots, p_n}(p),f^{\{n\}}_{p_1,\ldots, p_n}\left(\widetilde{\mathcal{T}}_{(f^{\{n\}}_{p_1,\ldots, p_n}), p} (q)\right)\right) \leq \rho_{\mathbb{B}} (p, \mathcal{T}_p (q)),
\end{equation*}
 where $\mathcal{T}_p(q)=(1-qp)^{-1}q(1-qp)$ and
 \[ \widetilde{\mathcal{T}}_{(f^{\{n\}}_{p_1,\ldots, p_n}),p} (q) \assign h(q)^{-1} q h(q),\quad h(q)= 1-(f^{\{n\}}_{p_1,\ldots, p_n}(p)\ast (f^{\{n\}}_{p_1,\ldots, p_n})^c)(q). \]
Let us observe that in the case $q\in\mathbb R\cap\mathbb B=(-1,1)$ the maps $\mathcal{T}_p$ and $\widetilde{\mathcal{T}}_{(f^{\{n\}}_{p_1,\ldots, p_n}),p}$ act as identity on $q$ and we reduce to the classical Schwarz--Pick inequality in the pullback form:
\begin{equation*}
\rho_{\mathbb{B}} \left(f^{\{n\}}_{p_1,\ldots, p_n}(p),f^{\{n\}}_{p_1,\ldots, p_n}(q)\right) \leq \rho_{\mathbb{B}} (p, q).
\end{equation*}

\section{On Nevanlinna--Pick interpolation problem}\label{sec:np}

A result on Nevanlinna--Pick interpolation problem over the unit ball of $\mathbb B$ which is essentially analogous to the complex case has already been stated by Alpay, Bolotnikov, Colombo and Sabadini in \cite{alpaybolotnikovcolombosabadini}.

\begin{theorem}[{\cite[Thm. 1.3]{alpaybolotnikovcolombosabadini}}]\label{thm:np_quat_alpay}
Given $n$ distinct points $p_1,\ldots,p_n\in\mathbb B$ and $n$ values $s_1,\ldots,s_n\in{\mathbb B}$, there exists a function $f\in\mathcal{SR}(\mathbb B,{\mathbb B})$ such that $f(p_m)=s_m$ for all $m=1,\ldots, n$ if and only if the Hermitian matrix
\[
 P=\begin{pmatrix}
    \sum\limits_{\kappa=0}^{+\infty} p_m^\kappa (1-s_m \overline{s_\ell}) \overline{p_\ell}^\kappa
   \end{pmatrix}_{m,\ell=1}^n,
\]
namely the \emph{Pick matrix} associated to $\{p_i\}$ and $\{s_i\}$, is positive semidefinite.

\end{theorem}

\subsection{Two points Nevanlinna--Pick interpolation}\mbox{}

When $n=2$, in the complex case the Pick matrix is positive semidefinite if and only if the `Schwarz--Pick' condition holds:
\[
 \left|\dfrac{s_1-s_2}{1-s_1\overline{s_2}}\right|\leq  \left|\dfrac{p_1-p_2}{1-p_1\overline{p_2}}\right|.
\]
In the quaternionic framework we cannot expect an analogous result in general by the different expression of the Schwarz--Pick Lemma. However, if we assume that one of the points $p_1,p_2$ is \emph{real}, we can recover the same result.
\begin{proposition}\label{prop:np2}
 Let $r\in(-1,1)$, $p\in\mathbb B\setminus\{r\}$ and $s,q\in{\mathbb B}$. There exists a slice regular function $f\in\mathcal{SR}(\mathbb B,\mathbb B)$ such that $f(r)=s$ and $f(p)=q$ if and only if
 \begin{equation}\label{eq:nev2ineq}
  \rho_\mathbb B(q,s)\leq  \rho_\mathbb B(p,r).
 \end{equation}
\end{proposition}
 \begin{proof} 
  If there exists a function $f$ as required, then by \eqref{eq3prho} we recall that
  \[
\rho_{\mathbb{B}}(f(p), f(\widetilde{\mathcal T}_{f,p}(r)))\leq \rho_{\mathbb B}(p,\mathcal{T}_p(r)),
  \]
but $\widetilde{\mathcal T}_{f,p}(r)=\mathcal{T}_p(r)=r$ since $r$ is real: thus $ \rho_{\mathbb B} (q,s)=\rho_{\mathbb{B}}(f(p), f(r))\leq \rho_{\mathbb B}(p,r)$.
  
  Vice versa, let us assume $ \rho_\mathbb B(q,s)\leq  \rho_\mathbb B(p,r)$. Then, by noticing that
  \[
   1-\rho_{\mathbb B}(q,s)^2=1-|1-s \overline{q}|^{-2}|s-q|^2=|1-s\overline{q}|^{-2}(1+|s|^2|q|^2-|s|^2-|q|^2)=\dfrac{(1-|s|^2)(1-|q|^2)}{|1-s\overline{q}|^2},
  \]
  our assumption is equivalent to say that
  \begin{equation}\label{eq:sp_int1}
    \dfrac{(1-|s|^2)(1-|q|^2)}{|1-s\overline{q}|^2}\geq \dfrac{(1-|r|^2)(1-|p|^2)}{|1-r \overline p|^2}.
  \end{equation}
If we take the two pairs $(r,p)$ and $(s,q)$, the associated Pick matrix $P$ is such that
\begin{align*}
 P_{1,1} &=\sum\limits_{\kappa=0}^{+\infty} (1-|s|^2)r^\kappa \overline r^\kappa=(1-|s|^2)(1-|r|^2)^{-1}\\
 P_{2,2} &=\sum\limits_{\kappa=0}^{+\infty}(1-|q|^2)p^\kappa \overline p^\kappa=(1-|q|^2)(1-|p|^2)^{-1}\\
 P_{1,2} &=\sum\limits_{\kappa=0}^{+\infty}(1-s\overline{q})r^\kappa \overline p^\kappa=(1-s\overline{q})(1-r\overline p)^{-1}\\
 P_{2,1} &=\overline{P_{1,2}}=(1- p r)^{-1}(1-q\overline s).
\end{align*}
Clearly $P_{1,1}\geq 0$, so the matrix is positive semidefinite if and only if $P_{1,1}P_{2,2}-|P_{1,2}|^2\geq 0$, which is equivalent to \eqref{eq:sp_int1}. 
 \end{proof}

 Now let's assume that both $r$ and $p$ are real and $P\geq 0$. In this particular setting we can reproduce the argument of the complex case (see \cite{BaribeauRivardWegert} for instance) to find the solutions of the interpolation problem.
 
 \begin{proposition}
 Let $r,p \in(-1,1)$, $r\ne p$ and $s,q\in{\mathbb B}$, such that the Pick matrix associated to the pairs $(r,p)$ and $(s,q)$ is positive semidefinite, or equivalently \eqref{eq:nev2ineq} holds. If the Pick matrix is singular, then $Q\assign M_{r}(p)^{-1}M_{s}(q)$ is unimodular and the interpolation problem
 \[
f(r)=s,\quad f(p)=q,\qquad f\in\mathcal{SR}(\mathbb B,\mathbb B)
 \]
has a unique solution given by the M\"obius transformation
\begin{equation*}
 {f=\mathcal M_{-s}\bullet (\mathcal M_r\ast Q).}
\end{equation*}
Otherwise, if $P>0$, then $|Q|<1$ and the interpolation problem has infinite solutions parameterized by the space $\mathcal {SR}(\mathbb B,\overline{\mathbb B})=\mathcal{SR}(\mathbb B,\mathbb B)\cup\partial\mathbb B$ \footnote{This follows from the Maximum Modulus Principle in the quaternionic setting, see \cite[Thm. 7.1]{librospringer2}.}:
\begin{equation*}
 {f=\mathcal M_{-s}\bullet(\mathcal M_r\ast (\mathcal M_{-Q}\bullet (\mathcal M_p \ast h)))}\qquad (h\in\mathcal{SR}(\mathbb B,\mathbb B)\cup\partial\mathbb B)
\end{equation*}
\end{proposition}

\begin{proof} 
By hypothesis $P\geq 0$, hence by Theorem \ref{thm:np_quat_alpay}  there exists a solution in $\mathcal{SR}(\mathbb B,\mathbb B)$ for the interpolation problem $f(r)=s$ and $f(p)=q$. But since we assumed that $r$ and $p$ are real, then $f^{\star}_r(p)=(\mathcal M_{r}^{-\ast}\ast(\mathcal M_{f(r)}\bullet f))(p)$ can be computed as $M_{r}(p)^{-1}M_{f(r)}(f(p))=M_r(p)^{-1}M_s(q)=Q$.

Consider first the singular case, $P\geq 0$ and $\det P=0$: in that case $\rho_\mathbb B(s,q)=\rho_\mathbb B (r,p)$, so $f^{\star}_r(p)=Q\in\partial\mathbb B$, thus $ f^{\star}_r\equiv Q$ and $f^{\{2\}}_{r,p}\equiv Q$; therefore $f$ must be a M\"obius transformation for which $\mathcal M_r^{-\ast}\ast (\mathcal M_s\bullet f)= Q$. It is easy to check that $
  f=\mathcal M_{-s}\bullet (\mathcal M_r\ast Q) $
is the solution, as $f(r)=M_{-s}(0)=s$ and $f(p)=M_{-s}(M_r(p)Q)=M_{-s}(M_{s}(q))=q$.

In the non-singular case $P>0$, we have that $g\assign f^{\star}_r:\mathbb B\lto \mathbb B$ and sends $p$ to $Q\in\mathbb B$. Then $f^{\{2\}}_{r,p}=g^{\star}_p$ could be either a unimodular constant or a slice regular self-map of $\mathbb B$. It is straightforward to see that the problem
\begin{equation*}
 \begin{cases}
g^{\star}_p=\mathcal M_p^{-\ast}\ast(\mathcal M_{g(p)}\bullet g)=h\\
g(p)=Q
 \end{cases}
\end{equation*}
has the unique slice regular solution $g=\mathcal M_{-Q}\bullet (\mathcal M_p \ast h)$ for any $h\in\mathcal{SR}(\mathbb B,\mathbb B)\cup\partial \mathbb B$; once we fix such a function $h$, we can find $f$ so that
\begin{equation*}
 \begin{cases}
f^{\star}_r=g=\mathcal M_{-Q}\bullet (\mathcal M_p \ast h)\\
f(r)=s
 \end{cases}
\end{equation*}
in the same way, getting $f=\mathcal M_{-s}\bullet(\mathcal M_r\ast g)=\mathcal M_{-s}\bullet(\mathcal M_r\ast (\mathcal M_{-Q}\bullet (\mathcal M_p \ast h)))$. By direct computation, $f(p)=M_{-s}(M_r(p)g(p))=M_{-s}(M_r(p)Q)=M_{-s}(M_s(q))=q$.
 \end{proof}

\begin{example}
 Let $\lambda,\mu\in(-1,1)$. We want to find conditions on these parameters such that there exists a slice regular map $f$ from $\mathbb B$ to itself for which $f\left(-\frac{1}{2}\right)=\lambda i$ and $f\left(\frac{1}{2}\right)=\mu j$. Then it is enough to compute
 \[
Q=M_{-\frac{1}{2}}\left(\frac{1}{2}\right)^{-1}\cdot M_{\lambda i}(\mu j)=\frac{5}{4}(1-\lambda \mu k)^{-1}(\mu j -\lambda i) \implies |Q|^2=\frac{25}{16}\cdot\frac{\mu^2+\lambda^2}{1+\lambda^2\mu^2}.
 \]
 So the condition for the existence of a solution for the interpolation problem $|Q|\leq 1$ in this case becomes $(25-16\lambda^2)\mu^2\leq 16-25 \lambda^2$. In particular, if we fix $\lambda$ so that $1>|\lambda|>\frac{4}{5}$, we have no solutions for any choice of $\mu$; if otherwise $\lambda$ is fixed so that $0\leq |\lambda|\leq\frac{4}{5}$, then we have solutions once
 \[
 |\mu|\leq \sqrt{\frac{(4+5\lambda)(4-5\lambda)}{(5+4\lambda)(5-4\lambda)}}.
 \]
\end{example}
When equality holds ($|Q|=1$), we have the extremal case of a unique (M\"obius) solution, otherwise we have infinite solutions parameterized by $\mathcal{SR}(\mathbb B,{\mathbb B})\cup \partial \mathbb B$.

For instance, let $\lambda=\frac{1}{4}$. Then we have solutions if and only if $|\mu|\leq\sqrt{\frac{77}{128}}$. When the inequality is strict, solutions are given (implicitly) by the parametric formula
\[
f=\mathcal M_{-\frac{1}{4}i}\bullet(\mathcal M_{-\frac{1}{2}}\ast (\mathcal M_{-Q}\bullet (\mathcal M_{\frac{1}{2}} \ast h))),\quad Q=\frac{5}{4}(4-\mu k)^{-1}(-i+4\mu j),\quad h\in\mathcal{SR}(\mathbb B,{\mathbb B})\cup\partial \mathbb B.
\]
When $\mu=\pm\sqrt{\frac{77}{128}}$, the unique solution is $f=\mathcal M_{-\frac{1}{4}i}\bullet(\mathcal M_{-\frac{1}{2}}\ast Q)$.

\subsection{Nevanlinna--Pick interpolation with three points}\mbox{}

For the case $n\geq 3$, in the complex setting, a result which was stated by Beardon and Minda \cite[Thm. 7.1]{BeardonMinda04} in the case $n=3$ and then generalized by Baribeau, Rivard and Wegert \cite[Thm. 4.1--4.2]{BaribeauRivardWegert} gives a more geometric set of conditions which are equivalent to the existence of solutions for the Nevanlinna--Pick interpolation problem. 
If we assume that the nodes in the domain are \emph{real}, we get exactly the same conditions also in our framework.

For the sake of simplicity, we first see the case $n=3$ before going to the general picture.

\begin{theorem}\label{thm:3point-nev-pick}
 Let $r_1,r_2,r_3\in(-1,1)$ be distinct and let $s_1,s_2,s_3\in\mathbb B$. Consider the interpolation problem
 \[
  f(r_m)=s_m\quad\text{for all }m=1,2,3,\qquad f\in\mathcal{SR(\mathbb B,\mathbb B)}.
 \]
 Let us define
 \begin{align*}
  Q_1^2 &\assign M_{r_1}(r_2)^{-1}M_{s_1}(s_2)\\
  Q_1^3 &\assign M_{r_1}(r_3)^{-1}M_{s_1}(s_3)
 \end{align*}
and \footnote{Here $\infty$ is just a conventional definition, it does not correspond to a pole of the ratio. For our purposes, we could have chosen any quaternion of norm $>1$. The other cases correspond to having as $Q_1^2,Q_1^3$ two different points on $\partial\mathbb B$, or a point inside $\mathbb B$ and the other one outside, or at least one of the two points not in $\overline{\mathbb B}$.}
\begin{equation*}
 Q_2^3 \assign 
    \begin{cases}
    M_{r_2}(r_3)^{-1}M_{Q_1^2}(Q_1^3) &\text{ if }Q_1^2,Q_1^3\in\mathbb B,\\
    Q_1^2 &\text{ if }Q_1^2=Q_1^3\in\partial \mathbb B,\\
    \infty &\text{ otherwise}.
    \end{cases} 
\end{equation*}
\begin{itemize}
\item If $|Q_2^3|<1$, we have an infinite family of solutions \emph{(non-singular case)} given by
\begin{equation}\label{eq:solution3_1}
 {\mathcal M_{-s_1}\bullet (\mathcal{M}_{r_1}\ast(\mathcal{M}_{-Q_1^2}\bullet(\mathcal{M}_{r_2}\ast(\mathcal{M}_{-Q_2^3}\bullet(\mathcal{M}_{r_3}\ast h)))))}
\end{equation}
where $h$ is a function in $\mathcal{SR}(\mathbb B,{\mathbb B})$ or a unimodular constant. In the last situation, the solution is a regular Blaschke product of degree three.
\item If $Q_2^3\in\partial\mathbb B$, there are two cases.
\begin{enumeratealpha}
	\item If $Q_1^2,Q_1^3$ both belong to $\mathbb B$, there exists a unique solution \emph{(singular case)} given by the regular Blaschke product of degree two
	\begin{equation}\label{eq:solution3_2}
		{\mathcal{M}_{-s_1}\bullet(\mathcal M_{r_1}\ast(\mathcal{M}_{-Q_1^2}\bullet(\mathcal M_{r_2}\ast Q_2^3)))}
	\end{equation}
	\item When $Q_1^2=Q_1^3\in\partial\mathbb B$, the solution is given by the regular M\"obius map
	\begin{equation}\label{eq:solution3_3}
		{\mathcal{M}_{-s_1}\bullet(\mathcal{M}_{r_1}\ast Q_1^2)}
	\end{equation}
\end{enumeratealpha}
\item Otherwise, if $Q_2^3$ has modulus strictly greater than one, the interpolation problem has no solutions.
\end{itemize}
\end{theorem}

\begin{proof} 
We first observe that if we have a (slice regular) solution $f:\mathbb B\lto \mathbb B$ of the interpolation problem $f(r_m)=s_m, m=1,2,3$, then
\begin{equation}\label{eq:Q_interp}
Q_1^2 = f^{\star}_{r_1}(r_2),\quad Q_1^3=f^{\star}_{r_1}(r_3),\quad Q_2^3=f^{\{2\}}_{r_1,r_2}(r_3).
\end{equation}
Indeed $f^{\star}_{r_1}\in\mathcal{SR}(\mathbb B,{\mathbb B})\cup\partial \mathbb B$, hence $\overline{\mathbb B}\ni f^{\star}_{r_1}(r_m)=\mathcal M_{r_1}^{-\ast}\ast(\mathcal{M}_{s_1}\bullet f)\mid_{r_m}$, $m=2,3$: since $r_m$ is real, the last term is equal to $\mathcal M_{r_1}(r_m)^{-1} (\mathcal M_{s_1}\bullet f)(r_m)=M_{r_1}(r_m)^{-1} M_{s_1}(f(r_m))=M_{r_1}(r_m)^{-1} M_{s_1}(s_m)=Q_1^m$ for $m=2,3$. If $f$ is a M\"obius transformation, then $f^{\star}_{r_1}=f^{\{2\}}_{r_1,r_2}\equiv \eta\in\partial \mathbb B$, so we have also $Q_1^2=Q_1^3=Q_2^3=\eta$. Otherwise, $f^{\star}_{r_1}$ has values in the open unit ball and $f^{\{2\}}_{r_1,r_2}(r_3)=(f^{\star}_{r_1})^{\star}_{r_2}(r_3)$, which as above can be computed as $M_{r_2}^{-1}(r_3) M_{f^{\star}_{r_1}(r_2)}(f^{\star}_{r_1}(r_3))=M_{r_2}^{-1}(r_3) M_{Q_1^2}(Q_1^3)=Q_2^3$.

This claim shows that we cannot have solutions in the case $Q_2^3$ has modulus strictly greater than one.

Let us assume $|Q_2^3|<1$. This implies by its definition that also $Q_1^2,Q_1^3$ lie within the open unit ball. Hence by \eqref{eq:Q_interp}, if $f$ exists, then both $f^\star_{r_1}$ and $f^{\{2\}}_{r_1,r_2}$ are slice regular maps from $\mathbb B$ to itself. Moreover we can define $f^{\{3\}}_{r_1,r_2,r_3}\in\mathcal{SR(\mathbb B,\mathbb B)}\cup\partial\mathbb B$.  Now, let's take $h$ arbitrarily in this latter set of functions and seek for a function $g_1$ whose hyperbolic difference quotient at $r_3$ is $h$ and attains the value $Q_2^3$ at $r_3$, i.e., the solution of
\begin{equation*}\begin{cases}
		(g_1)^{\star}_{r_3}=h\\ 
		g_1(r_3)=Q_2^3\end{cases}.
\end{equation*}
But $(g_1)^{\star}_{r_3}=h\iff \mathcal M_{g_1(r_3)}\bullet g_1=\mathcal M_{r_3}\ast h$, hence the problem above has a unique (slice regular) solution given by $g_1=\mathcal{M}_{-Q_2^3}\bullet(\mathcal{M}_{r_3}\ast h)$. Now we proceed backwards, finding $g_2$ such that
\begin{equation*}\begin{cases}
		(g_2)^{\star}_{r_2}=g_1\\ g_2(r_2)=Q_1^2\end{cases}\implies g_2=\mathcal{M}_{-Q_1^2}\bullet (\mathcal M_{r_2}\ast g_1).
\end{equation*}
By \eqref{eq:Q_interp} we should also have $g_2(r_3)=Q_1^3$, but evaluating the solution above at $r_3$ one directly gets $g_2(r_3)=M_{-Q_1^2}(M_{r_2}(r_3)\cdot Q_2^3)=M_{-Q_1^2}(M_{Q_1^2}(Q_1^3))=Q_1^3$. It remains to find $g_3$ such that
\begin{equation*}\begin{cases}
		(g_3)^{\star}_{r_1}=g_2\\ g_3(r_1)=s_1\end{cases}\implies g_3=\mathcal{M}_{-s_1}\bullet (\mathcal M_{r_1}\ast g_2).
\end{equation*}
In the end, we get a function $f=g_3$, depending on the choice of $h$, which is slice regular from $\mathbb B$ to itself and satisfies $f(r_1)=s_1$; moreover, by direct evaluation, for $m=2,3$ it is $f(r_m)=M_{-s_1}(M_{r_1}(r_m)\cdot g_2(r_m))=M_{-s_1}(M_{r_1}(r_m)\cdot Q_1^m)=M_{-s_1}(M_{s_1}(s_m))=s_m$. Henceforth, it satisfies the interpolation conditions.

We can summarize this algorithm by the following scheme (we indicate with bold symbols the conditions we impose at each step, the others follow automatically):
{%
	\newcommand{\mc}[3]{\multicolumn{#1}{#2}{#3}}
	\begin{center}
		\begin{tabular}{cccl}
			$r_1$ & $r_2$ & $r_3$ & \\\cline{1-3}
			$\boldsymbol{s_1}$ & $s_2$ & \mc{1}{c|}{$s_3$} & $f=g_3$\\
			& $\boldsymbol{Q_1^2}$ & \mc{1}{c|}{$Q_1^3$} & $f^{\star}_{r_1}=g_2$\\
			&  & \mc{1}{c|}{$\boldsymbol{Q_2^3}$} & $f^{\{2\}}_{r_1,r_2}=g_1$\\
			&  & \mc{1}{c|}{} & $f^{\{3\}}_{r_1,r_2,r_3}=h$
		\end{tabular}
	\end{center}
}%

Now suppose $|Q_2^3|=1$, i.e., it is a point in the unit sphere $\partial\mathbb B$. There are two cases:
\begin{enumeratealpha}
	\item $Q_1^2$, $Q_1^3$ both lie in $\mathbb B$: if $f$ exists, by \eqref{eq:Q_interp} it is $f^\star_{r_1}\in\mathcal{SR}(\mathbb B,\mathbb B)$, while $f^{\{2\}}_{r_1,r_2}$ is constant and equal to $Q_2^3$. So this time we start by finding the slice regular function whose hyperbolic difference quotient at $r_{2}$ is $Q_2^3$ and assumes the value $Q_1^2$ at $r_2$, namely $\widetilde{g}=\mathcal M_{-Q_1^2}\bullet(\mathcal M_{r_2}\ast Q_2^3)$: by direct evaluation and since in our case $Q_2^3=M_{r_2}(r_3)$ it also satisfies $\widetilde{g}(r_3)=Q_1^3$. Since it is the result of a M\"obius transformation acting on another M\"obius transformation by $\bullet$, the map $\widetilde g$ is M\"obius too. Finally we find $f$ for which $f^\star_{r_1}=\widetilde{g}$ and $f(r_1)=s_1$, that is to say the second-degree regular Blaschke product $f=\mathcal M_{-s_1}\bullet(\mathcal M_{r_1}\ast \widetilde g)$, which also satisfies $f(r_2)=s_2$ and $f(r_3)=s_3$.
	\item $Q_1^2=Q_1^3\in\partial\mathbb B$. In this latter case, only the last step of the algorithm is needed, since $f^\star_{r_1}$ should be the unimodular constant $u=Q_1^2=Q_1^3$. Then $f=\mathcal M_{-s_1}\bullet (\mathcal M_{r_1}\ast u)$, which is now a M\"obius transformation, does the job by the same arguments seen above.\qedhere
\end{enumeratealpha}

\end{proof}

\begin{remark}
	We observe that the terms $Q_\kappa^\ell$ of the algorithm depend on the order we fix for the nodes of interpolation $r_1,r_2,r_3$, but the problem itself should not depend on that order. Hence, although $Q_{2}^3$ could change, the property of being inside, outside or in the boundary of the unit ball $\mathbb B$ is preserved. Moreover, a permutation of nodes produces a different expression for the family of solutions \eqref{eq:solution3_1} in the non-singular case and the unique solution \eqref{eq:solution3_2}, or \eqref{eq:solution3_3}, in the singular one: this is not a contradiction as those representations in terms of $\ast$-products and $\bullet$-actions are not unique.
\end{remark}
\begin{example}
 Let $\lambda,\mu\in(-1,1)$. For which value of these parameters there exists a slice regular map $f$ from $\mathbb B$ to itself for which $f(0)=0$, $f\left(-\frac{1}{2}\right)=\lambda i$ and $f\left(\frac{1}{2}\right)=\mu j$? 
 
 As seen before, we can solve this problem by computing:
 \begin{align*}
  Q_1^2 &=M_0\left(-\frac{1}{2}\right)^{-1}\cdot M_0(\lambda i)=-2\lambda i\\
  Q_1^3 &=M_0\left(\frac{1}{2}\right)^{-1}\cdot M_0(\mu j)=2\mu j\\
  Q_2^3 &=\begin{cases} M_{-\frac{1}{2}}\left(\frac{1}{2}\right)^{-1}\cdot M_{Q_1^2}(Q_1^3)=\frac{5}{2}(1+4\lambda \mu k)^{-1}(\lambda i+\mu j) \quad\text{if }|\lambda|,|\mu|< \frac{1}{2}\\ \infty\quad\text{otherwise}.\end{cases}
 \end{align*}
 The case $Q_2^3=Q_1^2\in\partial\mathbb B$ never happens since $Q_1^3$ is different from $Q_1^2$ for any choice of $\lambda$ and $\mu$, unless $\lambda=\mu=0$, but then they are equal to $0$ and they're not in the boundary of $\mathbb B$.
 
Now $|Q_2^3|^2=\frac{25}{4}\cdot \frac{\lambda^2+\mu^2}{1+16\lambda^2\mu^2}$. So $|Q_2^3|\leq1$ if and only if $(25-64\lambda^2)\mu^2\leq 4-25 \lambda^2$ for $\lambda,\mu\in(-\frac{1}{2},\frac{1}{2})$. Suppose $\lambda$ is fixed so that $\frac{1}{2}>|\lambda|>\frac{2}{5}$: then we have no solutions for any choice of $\mu$; if otherwise we fix $\lambda$ such that $0\leq |\lambda|\leq\frac{2}{5}$, then $25-64\lambda^2>0$ and we conclude that the interpolation problem has solutions for
 \[
 |\mu|\leq \sqrt{\frac{(2+5\lambda)(2-5\lambda)}{(5+8\lambda)(5-8\lambda)}}.
 \]
\end{example}
The problem is singular when the inequalities above are equalities. 

For instance, if $\lambda=1/4$, then we can find solutions once $|\mu|\leq\sqrt{\frac{13}{112}}=\frac{1}{4}\sqrt{\frac{13}{7}}$. If the strict inequality holds, there is a family of solutions of the form
\begin{equation*}
q\ast(\mathcal{M}_{\frac{i}{2}}\bullet(\mathcal{M}_{-\frac{1}{2}}\ast(\mathcal{M}_{-Q_2^3}\bullet(\mathcal{M}_{\frac{1}{2}}\ast h)))), \quad Q_2^3=\frac{5}{2}(1+\mu k)^{-1}\cdot\left(\frac{1}{4} i+\mu j\right),\quad h\in\mathcal{SR}(\mathbb B,{\mathbb B})\cup \partial\mathbb B,
\end{equation*}
while in the singular cases $\mu=\pm\sqrt{\frac{13}{112}}$ the unique solution is the regular Blaschke product of degree two $q\ast(\mathcal{M}_{\frac{i}{2}}\bullet(\mathcal{M}_{-\frac{1}{2}}\ast Q_2^3))$.
\begin{corollary}\label{cor:np_dist}
 Let $r_1,r_2,r_3$ be pairwise distinct points in $(-1,1)$ and let $s_1,s_2,s_3\in\mathbb B$. The interpolation problem $f(r_m)=s_m, m=1,2,3$ with $f\in\mathcal{SR}(\mathbb B,\mathbb B)$ has infinite solutions if and only if
 \begin{equation}\label{eq:1np3}
  \rho_\mathbb B(s_1,s_m)<\rho_\mathbb B(r_1,r_m)\qquad \text{for } m=2,3
 \end{equation}
and
\begin{equation}\label{eq:2np3}
 \rho_\mathbb B(M_{r_1}(r_2)^{-1} M_{s_1}(s_2), M_{r_1}(r_3)^{-1} M_{s_1}(s_3))<\rho_\mathbb B(r_2,r_3).
\end{equation}

If equality holds in \eqref{eq:2np3}, we have a unique solution given by a regular Blaschke product of degree two.

\end{corollary}
Inequalities \eqref{eq:1np3} imply that the two points in the left hand side of \eqref{eq:2np3} both lie in the open unit ball, so it makes sense to consider their pseudo-hyperbolic distance in $\mathbb B$.

\begin{proof} 
 By Theorem \ref{thm:3point-nev-pick}, the interpolation problem has infinite solutions if and only if $|Q_2^3|<1$, and that's true if and only if $Q_1^2,Q_1^3\in\mathbb B$ and $|M_{Q_1^2}(Q_1^3)|<|M_{r_2}(r_3)|$, which correspond to the strict inequalities \eqref{eq:1np3} and \eqref{eq:2np3}.
 
 If \eqref{eq:1np3} holds and we have equality in place of strict inequality in \eqref{eq:2np3}, then this corresponds to $Q_1^2,Q_1^3\in\mathbb B$ with $|Q_2^3|=1$ and we still conclude by Theorem \ref{thm:3point-nev-pick}.
\end{proof}

\subsection{Nevanlinna--Pick interpolation with more than three points}\mbox{}

The same procedure by iteration seen above could be used to deal with the case $n>3$, as illustrated in \cite{BaribeauRivardWegert} for the complex case.

\begin{theorem}\label{thm:np_npoints}
  Let us assume $n\geq 3$, let $r_1,r_2,\ldots,r_n\in(-1,1)$ be distinct and let $s_1,s_2,\ldots s_n\in\mathbb B$. Consider the interpolation problem
 \begin{equation}\label{eq:interpolationproblem}
  f(r_m)=s_m\quad\text{for all }m=1,\ldots, n,\qquad f\in\mathcal{SR(\mathbb B,\mathbb B)}.
 \end{equation}
 
 Let 
 \begin{align*}
 	Q_0^m \assign s_m \qquad\text{for all }m=1,\ldots, n\\
  Q_1^m \assign M_{r_1}(r_m)^{-1}M_{s_1}(s_m) \qquad\text{for all }m=2,\ldots, n
 \end{align*}
and for any $\kappa=2,\ldots,n-1$, $\ell=\kappa+1,\ldots,n$
\begin{equation*}
 Q_\kappa^\ell \assign 
    \begin{cases}
    M_{r_\kappa}(r_\ell)^{-1}M_{Q_{\kappa-1}^\kappa}(Q_{\kappa-1}^\ell) &\text{ if }Q_{\kappa-1}^\kappa,Q_{\kappa-1}^\ell\in\mathbb B,\\
    Q_{\kappa-1}^\kappa &\text{ if }Q_{\kappa-1}^\kappa=Q_{\kappa-1}^\ell\in\partial \mathbb B,\\
    \infty &\text{ otherwise}.
    \end{cases}
\end{equation*}
The interpolation problem is non-singular, with infinite solutions, when $|Q_{n-1}^n|<1$. It is singular, with a unique (regular Blaschke) solution, if $|Q_{n-1}^n|=1$, whose degree is $\kappa_0\in\{1,\ldots,n-1\}$ if and only if $Q_{\kappa_0}^{\kappa_0+1}=\ldots=Q_{\kappa_0}^n\in\partial\mathbb B$, but $Q_{\kappa_0-1}^\ell\in\mathbb B$ for some (hence any) $\ell\in\{\kappa_0,\kappa_0+1,\ldots,n\}$.  Otherwise, the problem has no admissible solutions.
\end{theorem}

\begin{proof} 
The proof is essentially the same as in the case of three points, just with more iterations of the algorithm, which we can illustrate by the following scheme.
{%
\newcommand{\mc}[3]{\multicolumn{#1}{#2}{#3}}
\begin{center}
\begin{tabular}{ccccccl}
$r_1$ & $r_2$ & $r_3$ & $\cdots$ & $r_{n-1}$ & $r_n$ & \\\cline{1-6}
$\boldsymbol{s_1}$ & $s_2$ & $s_3$ & $\cdots$ & $s_{n-1}$ & \mc{1}{c|}{$s_n$} & $f=g_n$\\
 & $\boldsymbol{Q_1^2}$ & $Q_1^3$ & $\cdots$ & $Q_1^{n-1}$ & \mc{1}{c|}{$Q_1^n$} & $f^{\star}_{r_1}=g_{n-1}$\\
 &  & $\boldsymbol{Q_2^3}$ & $\cdots$ & $Q_2^{n-1}$ & \mc{1}{c|}{$Q_2^n$} & $f^{\{2\}}_{r_1,r_2}=g_{n-2}$\\
 &  &  & $\boldsymbol{\ddots}$ & $\vdots$ & \mc{1}{c|}{$\vdots$} & $\vdots$\\
 &  &  &  &  $\boldsymbol{Q_{n-2}^{n-1}}$ & \mc{1}{c|}{$Q_{n-2}^n$} & $f^{\{n-2\}}_{r_1,\ldots,r_{n-2}}=g_2$\\
 &  &  &  &  &  \mc{1}{c|}{$\boldsymbol{Q_{n-1}^n}$} & $f^{\{n-1\}}_{r_1,\ldots,r_{n-1}}=g_1$\\
 &  &  &  &  & \mc{1}{c|}{} & $f^{\{n\}}_{r_1,\ldots,r_n}=h$
\end{tabular}
\end{center}
}%

Analogously to the three-point case, the existence of an admissible interpolating function $f$ implies that 
\begin{equation}\label{eq:interp_n}
f^{\{\kappa\}}_{r_1,\ldots,r_\kappa}(r_\ell)=Q_\kappa^\ell\qquad \kappa\in\{1,\ldots,n-1\},\ell\in\{\kappa+1,\ldots,n\}.
\end{equation}
This forces in particular $|Q_{n-1}^n|\leq 1$ in order to have a solution.

Assuming $|Q_{n-1}^n|<1$, by the definition of the coefficients in the scheme, all of them must belong to the unit ball $\mathbb B$, and then \eqref{eq:interp_n} implies that for a slice regular interpolating function $f:\mathbb B\lto \mathbb B$ -- if existing -- we have $f^{\{n-1\}}_{r_1,\ldots, r_{n-1}}:\mathbb B\lto \mathbb B$ slice regular and $f^{\{n\}}_{r_1,\ldots,r_n}\in\mathcal{SR}(\mathbb B,\mathbb B)$. We proceed following the scheme from the bottom to the top: given $h\in\mathcal{SR}(\mathbb B,{\mathbb B})$ or $h\equiv u\in\partial\mathbb B$, we find $g_1$ such that $(g_1)^{\star}_{r_n}=h$ and $g_1(r_n)=Q^{n}_{n-1}$, then $g_2$ for which $(g_2)^{\star}_{r_{n-1}}=g_1$ and $g_2(r_{n-1})=Q_{n-2}^{n-1}$: the solution will also satisfy $g_2(r_n)=Q_{n-2}^n$. Following the rows, in the end one gets a slice regular function $f=g_n$ satisfying \eqref{eq:interp_n} and thus $f(r_m)=s_m$ for each $m=1,\ldots,n$, depending on the choice of the parameter $h$. 

Namely, we get \[{f=g_n=\mathcal M_{-s_1}\bullet(\mathcal M_{r_1}\ast (\overbrace{\mathcal M_{-Q_1^2}\bullet(\mathcal M_{r_2}\ast(\ldots\ast(\underbrace{\mathcal M_{-Q_{n-1}^n}\bullet(\mathcal M_{r_n}\ast h)}_{g_1})))}^{g_{n-1}}))}\]

If now we assume $|Q_{n-1}^n|=1$, i.e, $Q_{n-1}^n=u\in\partial\mathbb B$ this means -- as a consequence of the definition of the terms $Q_\kappa^\ell$ -- that all $Q_\kappa^\ell=u$ for all indices $\kappa$ greater or equal some $\kappa_0\in\{1,\ldots,n-1\}$, i.e., all the terms from the $(\kappa_0+1)$-th row of the scheme to the bottom. Suppose that the $\kappa_0$-th row contains a term within the unit ball: since the terms $Q_{\kappa_0}^\ell$ are not $\infty$, all the $\kappa_0$-th row consists of points of $\mathbb B$. By \eqref{eq:interp_n}, for an interpolating function $f$ we have that $f^{\star}_{r_1}, f^{\{2\}}_{r_1,r_2},\ldots f^{\{\kappa_0-1\}}_{r_1,\ldots,r_{\kappa_0-1}}\in\mathcal{SR}(\mathbb B,\mathbb B)$ while $f^{\kappa_0}_{r_1,\ldots,r_{\kappa_0}}\equiv u\in\partial\mathbb B$. So we start from that row and go upwards in the scheme: we find $g_{n-\kappa_0+1}$ such that $(g_{n-\kappa_0+1})^{\star}_{r_{\kappa_0}}=g_{n-\kappa_0}\equiv u $ and $g_{n-\kappa_0+1}(r_{\kappa_0})=Q_{\kappa_0-1}^{\kappa_0}$, i.e., $g_{n-\kappa_0+1}=\mathcal M_{-Q_{\kappa_0-1}^{\kappa_0}}\bullet(\mathcal M_{r_{\kappa_0}}\ast u)$ and so on, until we get 
\begin{equation*}
 {f= g_n=\mathcal M_{-s_1}\bullet(\mathcal M_{r_1}\ast (\overbrace{\mathcal M_{-Q_1^2}\bullet(\mathcal M_{r_2}\ast(\ldots\ast(\underbrace{\mathcal M_{-Q_{\kappa_0-1}^{\kappa_0}}\bullet(\mathcal M_{r_{\kappa_0}}\ast u)}_{g_{n-\kappa_0+1}})))}^{g_{n-1}})).}
\end{equation*}

which is a regular Blaschke product of degree $\kappa_0$ by Proposition \ref{prop:blaschke_deg}; moreover it satisfies \eqref{eq:interp_n} and the interpolation condition $f(r_m)=s_m$, $m=1,\ldots,n$.

\end{proof}

The condition for having an infinite family of solutions (non-singular case) for the Pick interpolation problem is $|Q_{n-1}^n|<1$, or equivalently $|Q_\kappa^\ell|<1$ for all $\kappa=1,\ldots,n-1$ and $\ell=\kappa+1,\ldots,n$. This can be easily restated in terms of the pseudo-hyperbolic distance $\rho_\mathbb B$, analogously as in Corollary~\ref{cor:np_dist}.

\begin{corollary}
	The interpolation problem \eqref{eq:interpolationproblem} has infinite solutions if and only if \[\rho_\mathbb B(Q_{n-2}^{n-1},Q_{n-2}^n)<\rho_\mathbb{B}(r_{n-1},r_n),\] or equivalently, if $\rho_\mathbb B(Q_{\kappa-1}^{\kappa},Q_{\kappa-1}^\ell)<\rho_\mathbb{B}(r_{\kappa},r_\ell)$ for all $\kappa\in\{1,\ldots,n-1\},\ell\in\{\kappa+1,\ldots,n\}$ -- where we use the notation $Q_0^\ell\assign s_\ell$.
\end{corollary}
%

\subsection{The case of nodes not in the real line}\mbox{}

Since now we have considered a particular case for the quaternionic Nevanlinna--Pick interpolation, when all the nodes are \emph{real}. One may ask if Theorem \ref{thm:np_npoints} still works by replacing $r_1,\ldots,r_n\in(-1,1)$ with distinct points $p_1,\ldots p_n$ in the whole unit ball.

	In generality, the argument shown above breaks down. Indeed, we would like to define from the data of the interpolation problem $\{p_m\},\{s_m\}$ quantities $Q_\kappa^\ell$ such that $f^{\star}_{p_1}(p_\ell)=Q_1^\ell$ ($\ell=2,\ldots,n$) for any solution $f$ of the problem. But, assuming that $f$ exists,
	\[
	f^\star_{p_1}(p_\ell)=(\mathcal M_{p_1}^{-\ast}\ast\mathcal M_{s_1}\bullet f)(p_\ell)
	\]
	and the last term evaluates as $ M_{p_1}(\tilde p_\ell)^{-1}M_{s_1}(f(\hat p_\ell))$ where $\tilde p_\ell$ is a point in $\mathbb S_{p_\ell}$ depending on $p_1$ and $\hat p_\ell\in\mathbb S_{p_\ell}$ may depend on $p_1,s_1$ and on the conjugate function $f^c$. They trivially coincide with $p_\ell$ when that point is real. So, the proposed algorithm does not work anymore, as $\hat p_\ell$ generally depends on the function $f^c$ which is not known a priori and in any case $f(\hat p_\ell)$ is not known if $\hat p_\ell$ is not one of the nodes.
	
However, there is a specific situation in which we can recover a partial result in the spirit of Theorem \ref{thm:np_npoints}. As Theorem \ref{thm:3point-nev-pick} holds in the complex setting (\cite[Thm. 4.1]{BaribeauRivardWegert}) without the restriction to real nodes of interpolation, then we have the following partial answer: if $p_1,\ldots, p_n\in\mathbb B$ belong to a common slice $\mathbb {C}_I$ and also $s_1,\ldots,s_n\in\mathbb B\cap\mathbb {C}_I$, then the interpolation problem $f(p_m)=s_m$, $m\in\lbrace 1,\ldots, n\rbrace$ admits infinite solutions that are $\mathbb C_I$-preserving if $|Q_{n-1}^n|< 1$, while it admits at least a solution given by a regular Blaschke product of degree at most $n-1$, when $|Q_{n-1}^n|=1$, and this is the unique one preserving $\mathbb C_I$.
 
 Indeed, we can identify $\mathbb B\cap \mathbb C_I$ with the complex unit disk $\mathbb D$. Notice that $Q_{n-1}^n\in\mathbb C_I$ by our assumptions. If $|Q_{n-1}^n|\leq 1$, there exists a holomorphic function $f_0:\mathbb B\cap\mathbb{C}_I\lto\mathbb B\cap\mathbb {C}_I$ satisfying the interpolation properties.
 Now there exists a unique slice regular function $f:\mathbb B \lto \mathbb H$ extending $f_0$ by the formula (see \cite[Lemma 1.21]{librospringer2}):
 \begin{align*}
 	&f(x+yJ)=\dfrac{f_0(x-yI)+f_0(x+yI)}{2}+\dfrac{JI[f_0(x-yI)-f_0(x+yJ)]}{2}\\
 	&\forall\,x,y\in\mathbb R, x^2+y^2<1, J\in\mathbb S.
 \end{align*}
Fix $J$ as above and let $\lambda=\tmop{Re} (JI)$. Then $|\lambda|\leq1$ and there exists $K$ orthogonal to $I$ such that $JI=\lambda + \sqrt{1-\lambda^2} K$, so
\begin{align*}
	f(x+yJ)&=\dfrac{(1+\lambda)f_0(x-yI)+(1-\lambda)f_0(x+yI)}{2}+K \dfrac{\sqrt{1-\lambda^2}[f_0(x-yI)-f_0(x+yI)]}{2}\\
	&= G_0+ K G_1,\qquad G_0,G_1\in\mathbb B\cap\mathbb C_I.
\end{align*}
Hence we have an orthogonal splitting and by a simple calculation
\begin{equation*}
 |f(x+yJ)|^2=|G_0|^2+|G_1|^2=\left(\frac{1+\lambda}{2}\right)|f_0(x-yI)|^2+\left(\frac{1-\lambda}{2}\right)|f_0(x+yI)|^2.
\end{equation*}
The last term is clearly bounded above by $\max\{|f_0(x-yI)|^2,|f_0(x+yI)|^2\}$, which by our assumption is strictly less than $1$. This holds for all $J\in\mathbb S$ and therefore $f$ is a self map of the unit ball which is slice regular and solves the interpolation problem.

When $|Q_{n-1}^n|<1$, by \cite[Thm. 4.1]{BaribeauRivardWegert} $f_0$ varies in an infinite family of holomorphic functions, giving rise to infinite slice regular functions from $\mathbb B$ to $\mathbb B$ satisfying the requested interpolation. If $|Q_{n-1}^n|=1$, then $f_0$ is unique and given by a complex Blaschke product of degree $\kappa\in\lbrace 1,\ldots, n-1\rbrace$:
\[
f_0(\zeta)=\left(\prod_{\ell=1}^{\kappa}\frac{\zeta-q_\ell}{1-\zeta \overline{q_\ell}}\right)\cdot v,\quad \zeta,q_1,\ldots,q_\kappa\in\mathbb B\cap\mathbb C_I, v\in\partial\mathbb B.
\]
Since $q_1,\ldots,q_\kappa$ belong to the same slice $\mathbb C_I$, then the regular Blaschke product $B=\mathcal M_{q_1}\ast\ldots\ast \mathcal M_{q_\kappa}v$ is such that $B\mid_{\mathbb C_I}=f_0$: hence the unique extension of $f_0$ to a slice regular self-map of the unit ball $\mathbb B$ is exactly $B$ and we get that a solution of the interpolation problem is given by a regular Blaschke product of degree less or equal than $n-1$, which preserves the slice $\mathbb C_I$. Any solution preserving that slice must coincide with $B$ on the slice and hence everywhere.


\end{document}